\documentclass{amsart}
\pdfoutput=1

\usepackage{microtype}
\usepackage[pdftex,colorlinks,citecolor=black,linkcolor=black,urlcolor=black,bookmarks=false]{hyperref}
\usepackage{booktabs}
\usepackage{eucal}
\def\MR#1{\href{http://www.ams.org/mathscinet-getitem?mr=#1}{MR#1}}
\def\arXiv#1{arXiv:\href{http://arXiv.org/abs/#1}{#1}}
\usepackage{doi}
\usepackage{tikz}

\newtheorem{theorem}{Theorem}
\newtheorem{conjecture}[theorem]{Conjecture}
\newtheorem{lemma}[theorem]{Lemma}
\newtheorem{proposition}[theorem]{Proposition}
\newtheorem{corollary}[theorem]{Corollary}
\theoremstyle{definition}

\newtheorem{problem}[theorem]{Problem}
\numberwithin{theorem}{section}
\numberwithin{table}{section}
\numberwithin{figure}{section}
\numberwithin{equation}{section}

\DeclareMathOperator{\supp}{\mathrm{supp}}
\DeclareMathOperator{\vol}{\mathrm{vol}}
\newcommand{\R}{\mathbb{R}}
\newcommand{\C}{\mathbb{C}}
\newcommand{\Z}{\mathbb{Z}}
\newcommand{\bA}{\mathrm{A}}
\newcommand{\A}{\mathcal{A}}
\newcommand{\h}{\mathcal{H}}
\newcommand{\Schw}{\mathcal{S}_\textup{rad}}
\newcommand{\ft}{\widehat}
\renewcommand{\Re}{\mathop{\mathrm{Re}}}
\newcommand{\rhot}{\widetilde{\rho}}
\DeclareMathOperator{\rank}{\mathrm{rank}}

\title[An optimal uncertainty principle in twelve dimensions\dots]
{An optimal uncertainty principle\\ in twelve dimensions via modular forms}

\author{Henry Cohn and Felipe Gon\c{c}alves}

\date{February 20, 2019}

\address{Microsoft Research New England, One Memorial Drive, Cambridge, MA
02142, United States}

\email{cohn@microsoft.com}

\address{University of Alberta, Mathematical and Statistical Sciences, CAB 632,
Edmonton, Canada T6G 2G1}

\curraddr{Hausdorff Center for Mathematics, Universit\"at Bonn,
Endenicher Allee 60, 53115 Bonn, Germany}

\email{goncalve@math.uni-bonn.de}

\thanks{This work was begun during a visit by Gon\c{c}alves
to Microsoft Research New England.}

\begin{document}

\begin{abstract}
We prove an optimal bound in twelve dimensions for the uncertainty
principle of Bourgain, Clozel, and Kahane.  Suppose $f \colon \R^{12} \to
\R$ is an integrable function that is not identically zero.  Normalize its
Fourier transform $\ft f$ by $\ft f(\xi) = \int_{\R^d} f(x)e^{-2\pi i
\langle x, \xi\rangle}\, dx$, and suppose $\ft f$ is real-valued and
integrable.  We show that if $f(0) \le 0$, $\ft f(0) \le 0$, $f(x) \ge 0$
for $|x| \ge r_1$, and $\ft f(\xi) \ge 0$ for $|\xi| \ge r_2$, then $r_1r_2
\ge 2$, and this bound is sharp.  The construction of a function attaining
the bound is based on Viazovska's modular form techniques, and its
optimality follows from the existence of the Eisenstein series $E_6$. No
sharp bound is known, or even conjectured, in any other dimension. We also
develop a connection with the linear programming bound of Cohn and Elkies,
which lets us generalize the sign pattern of $f$ and $\ft f$ to develop a
complementary uncertainty principle.  This generalization unites the
uncertainty principle with the linear programming bound as aspects of a
broader theory.
\end{abstract}

\maketitle

\section{Introduction}

An uncertainty principle expresses a fundamental tradeoff between the
properties of a function $f$ and its Fourier transform $\ft f$.  The most
common variants measure the dispersion, with the tradeoff being that $f$ and
$\ft f$ cannot both be highly concentrated near the origin. Motivated by
applications to number theory, Bourgain, Clozel, and Kahane \cite{BCK} proved
an elegant uncertainty principle for the signs of $f$ and $\ft f$: if these
functions are nonpositive at the origin and not identically zero, then they
cannot both be nonnegative outside an arbitrarily small neighborhood of the
origin.  We can state this principle more formally as follows.

We say that a function $f\colon \R^d \to \R$ is \emph{eventually nonnegative}
(resp., \emph{nonpositive}) if $f(x)\geq 0$ (resp., $f(x)\leq 0$) for all
sufficiently large $|x|$. If that is the case, we let
\[
r(f)=\inf {} \{R \ge 0: \text{$f(x)$ has the same sign for $|x|\geq R$}\}
\]
be the radius of its last sign change.  We normalize the Fourier transform
$\ft f$ of $f$ by
\[
\ft f(\xi) = \int_{\R^d} f(x)e^{-2\pi i \langle x, \xi\rangle}\, dx.
\]
Let $\A_+(d)$ denote the set of functions $f\colon\R^d\to \R$ such that
\begin{enumerate}
\item $f\in L^1(\R^d)$, $\ft f\in L^1(\R^d)$, and $\ft f$ is real-valued
    (i.e., $f$ is even),
\item \label{cond:2} $f$ is eventually nonnegative while $\ft f(0)\leq 0$,
    and
\item \label{cond:3} $\ft f$ is eventually nonnegative while $f(0)\leq 0$.
\end{enumerate}
(Note the tension in \eqref{cond:2} between the eventual nonnegativity of $f$
and the inequality $\int_{\R^n} f = \ft f(0) \le 0$, and the analogous
tension in \eqref{cond:3}.)

The uncertainty principle of Bourgain, Clozel, and Kahane
\cite[Th\'eor\`eme~3.1]{BCK} says that
\[
\bA_+(d) := \inf_{f\in \A_+(d)\setminus\{0\}} \sqrt{r(f)r(\ft f\,)} >0.
\]
Taking the geometric mean of $r(f)$  and $r(\ft f\,)$ is a natural way to
eliminate scale dependence, because rescaling the input of $f$ preserves this
quantity. Thus, the uncertainty principle amounts to saying that $r(f)$ and
$r(\ft f\,)$ cannot both be made arbitrarily small if $f \in
\A_+(d)\setminus\{0\}$.

Gon\c calves, Oliveira e Silva, and Steinerberger \cite[Theorem~3]{GOS}
proved that for each dimension $d$ there exists a radial function
$f\in\A_+(d) \setminus\{0\}$ such that $f = \ft f$ and
\[
r(f) = r(\ft f\,) = \bA_+(d);
\]
furthermore, $\bA_+(d)$ is exactly the minimal value of $r(g)$ in the
following optimization problem:

\begin{problem}[$+1$ eigenfunction uncertainty principle]\label{+1problem}
Minimize $r(g)$ over all $g \colon \R^d \to \R$ such that
\begin{enumerate}
\item $g\in L^1(\R^d)\setminus\{0\}$ and $\ft g = g$, and
\item $g(0)=0$ and $g$ is eventually nonnegative.
\end{enumerate}
\end{problem}

The name ``$+1$ eigenfunction'' refers to the fact that $g$ is a
eigenfunction of the Fourier transform with eigenvalue $+1$.

Upper and lower bounds for $\bA_+(d)$ are known \cite{BCK,GOS}, but the exact
value has not previously been determined, or even conjectured, in any
dimension. Our main result is a solution of this problem in twelve
dimensions:

\begin{theorem}\label{thm_dimension12}
We have $\bA_+(12)=\sqrt{2}$. In particular, there exists a radial Schwartz
function $f\colon \R^{12} \to \R$ that is eventually nonnegative and
satisfies $\ft f=f$, $f(0)=0$, and
\[
r(f) = r(\ft f\,) = \sqrt{2}.
\]
Moreover, as a radial function $f$ has a double root at $|x|=0$, a single
root at $|x|=\sqrt{2}$, and double roots at $|x|=\sqrt{2j}$ for integers $j
\ge 2$.
\end{theorem}

See Figure~\ref{figure:dim12} for plots. The appealing simplicity of this
answer seems to be unique to twelve dimensions, and we have been unable to
conjecture a closed form for $\bA_+(d)$ in any other dimension $d$.  See
Section~\ref{section:numerics} for an account of the numerical evidence,
which displays noteworthy patterns and regularity despite the lack of any
exact conjectures.

We find the exceptional role of twelve dimensions surprising: why should a
seemingly arbitrary dimension admit an exact solution with mysterious
arithmetic structure not shared by other dimensions?  As far as we are aware,
Theorem~\ref{thm_dimension12} is the first time such behavior has arisen in
an uncertainty principle.

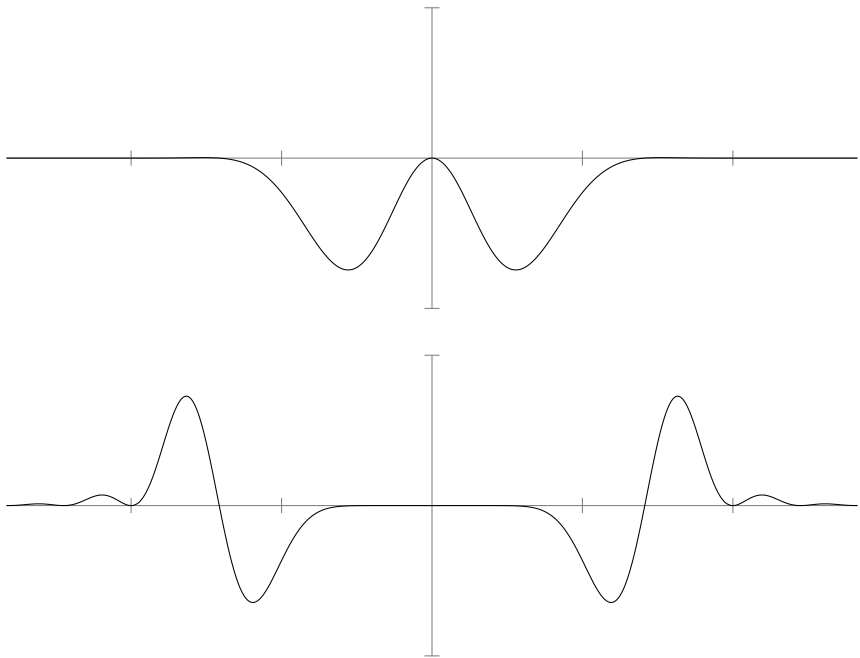
\begin{figure}
\begin{center}
\begin{tikzpicture}
\draw[black!50] (-5.65685,0) -- (5.65685,0);
\draw[black!50] (0,-2) -- (0,2);
\draw[black!50] (-0.1,2) -- (0.1,2);
\draw[black!50] (-0.1,-2) -- (0.1,-2);
\draw[black!50] (2,-0.1) -- (2,0.1);
\draw[black!50] (4,-0.1) -- (4,0.1);
\draw[black!50] (-2,-0.1) -- (-2,0.1);
\draw[black!50] (-4,-0.1) -- (-4,0.1);
\draw (-5.65685,0) -- (-5.64579,1.59528 E-10) -- (-5.63474,6.65799 E-10) -- (-5.62368,1.56060 E-9) -- (-5.61262,2.88572 E-9) -- (-5.60156,4.68256 E-9) -- (-5.59050,6.99164 E-9) -- (-5.57944,9.85213 E-9) -- (-5.56838,1.33014 E-8) -- (-5.55732,1.73743 E-8) -- (-5.54626,2.21029 E-8) -- (-5.53520,2.75159 E-8) -- (-5.52414,3.36376 E-8) -- (-5.51308,4.04882 E-8) -- (-5.50202,4.80823 E-8) -- (-5.49096,5.64291 E-8) -- (-5.47990,6.55315 E-8) -- (-5.46885,7.53857 E-8) -- (-5.45779,8.59807 E-8) -- (-5.44673,9.72978 E-8) -- (-5.43567,1.09310 E-7) -- (-5.42461,1.21983 E-7) -- (-5.41355,1.35273 E-7) -- (-5.40249,1.49126 E-7) -- (-5.39143,1.63482 E-7) -- (-5.38037,1.78270 E-7) -- (-5.36931,1.93410 E-7) -- (-5.35825,2.08814 E-7) -- (-5.34719,2.24383 E-7) -- (-5.33613,2.40014 E-7) -- (-5.32507,2.55592 E-7) -- (-5.31401,2.70996 E-7) -- (-5.30296,2.86100 E-7) -- (-5.29190,3.00769 E-7) -- (-5.28084,3.14867 E-7) -- (-5.26978,3.28252 E-7) -- (-5.25872,3.40779 E-7) -- (-5.24766,3.52304 E-7) -- (-5.23660,3.62681 E-7) -- (-5.22554,3.71769 E-7) -- (-5.21448,3.79427 E-7) -- (-5.20342,3.85522 E-7) -- (-5.19236,3.89929 E-7) -- (-5.18130,3.92531 E-7) -- (-5.17024,3.93224 E-7) -- (-5.15918,3.91917 E-7) -- (-5.14812,3.88538 E-7) -- (-5.13707,3.83029 E-7) -- (-5.12601,3.75359 E-7) -- (-5.11495,3.65516 E-7) -- (-5.10389,3.53517 E-7) -- (-5.09283,3.39406 E-7) -- (-5.08177,3.23259 E-7) -- (-5.07071,3.05185 E-7) -- (-5.05965,2.85328 E-7) -- (-5.04859,2.63870 E-7) -- (-5.03753,2.41034 E-7) -- (-5.02647,2.17083 E-7) -- (-5.01541,1.92323 E-7) -- (-5.00435,1.67107 E-7) -- (-4.99329,1.41833 E-7) -- (-4.98223,1.16944 E-7) -- (-4.97117,9.29320 E-8) -- (-4.96012,7.03376 E-8) -- (-4.94906,4.97481 E-8) -- (-4.93800,3.17983 E-8) -- (-4.92694,1.71696 E-8) -- (-4.91588,6.58863 E-9) -- (-4.90482,8.25370 E-10) -- (-4.89376,6.91131 E-10) -- (-4.88270,7.03572 E-9) -- (-4.87164,2.07443 E-8) -- (-4.86058,4.27333 E-8) -- (-4.84952,7.39469 E-8) -- (-4.83846,1.15351 E-7) -- (-4.82740,1.67929 E-7) -- (-4.81634,2.32675 E-7) -- (-4.80528,3.10587 E-7) -- (-4.79423,4.02662 E-7) -- (-4.78317,5.09885 E-7) -- (-4.77211,6.33223 E-7) -- (-4.76105,7.73617 E-7) -- (-4.74999,9.31971 E-7) -- (-4.73893,1.10914 E-6) -- (-4.72787,1.30594 E-6) -- (-4.71681,1.52309 E-6) -- (-4.70575,1.76127 E-6) -- (-4.69469,2.02104 E-6) -- (-4.68363,2.30287 E-6) -- (-4.67257,2.60715 E-6) -- (-4.66151,2.93410 E-6) -- (-4.65045,3.28384 E-6) -- (-4.63939,3.65635 E-6) -- (-4.62834,4.05145 E-6) -- (-4.61728,4.46877 E-6) -- (-4.60622,4.90781 E-6) -- (-4.59516,5.36786 E-6) -- (-4.58410,5.84801 E-6) -- (-4.57304,6.34718 E-6) -- (-4.56198,6.86406 E-6) -- (-4.55092,7.39712 E-6) -- (-4.53986,7.94462 E-6) -- (-4.52880,8.50461 E-6) -- (-4.51774,9.07489 E-6) -- (-4.50668,9.65305 E-6) -- (-4.49562,1.02365 E-5) -- (-4.48456,1.08222 E-5) -- (-4.47350,1.14073 E-5) -- (-4.46245,1.19884 E-5) -- (-4.45139,1.25621 E-5) -- (-4.44033,1.31246 E-5) -- (-4.42927,1.36721 E-5) -- (-4.41821,1.42006 E-5) -- (-4.40715,1.47059 E-5) -- (-4.39609,1.51839 E-5) -- (-4.38503,1.56301 E-5) -- (-4.37397,1.60401 E-5) -- (-4.36291,1.64097 E-5) -- (-4.35185,1.67342 E-5) -- (-4.34079,1.70094 E-5) -- (-4.32973,1.72310 E-5) -- (-4.31867,1.73950 E-5) -- (-4.30761,1.74973 E-5) -- (-4.29655,1.75343 E-5) -- (-4.28550,1.75027 E-5) -- (-4.27444,1.73994 E-5) -- (-4.26338,1.72220 E-5) -- (-4.25232,1.69684 E-5) -- (-4.24126,1.66371 E-5) -- (-4.23020,1.62273 E-5) -- (-4.21914,1.57391 E-5) -- (-4.20808,1.51731 E-5) -- (-4.19702,1.45311 E-5) -- (-4.18596,1.38158 E-5) -- (-4.17490,1.30310 E-5) -- (-4.16384,1.21818 E-5) -- (-4.15278,1.12744 E-5) -- (-4.14172,1.03166 E-5) -- (-4.13066,9.31752 E-6) -- (-4.11961,8.28814 E-6) -- (-4.10855,7.24096 E-6) -- (-4.09749,6.19040 E-6) -- (-4.08643,5.15279 E-6) -- (-4.07537,4.14652 E-6) -- (-4.06431,3.19213 E-6) -- (-4.05325,2.31242 E-6) -- (-4.04219,1.53260 E-6) -- (-4.03113,8.80302 E-7) -- (-4.02007,3.85774 E-7) -- (-4.00901,8.19284 E-8) -- (-3.99795,4.45085 E-9) -- (-3.98689,1.91891 E-7) -- (-3.97583,6.85746 E-7) -- (-3.96477,1.53054 E-6) -- (-3.95372,2.77391 E-6) -- (-3.94266,4.46663 E-6) -- (-3.93160,6.66273 E-6) -- (-3.92054,9.41950 E-6) -- (-3.90948,1.27975 E-5) -- (-3.89842,1.68608 E-5) -- (-3.88736,2.16765 E-5) -- (-3.87630,2.73154 E-5) -- (-3.86524,3.38514 E-5) -- (-3.85418,4.13618 E-5) -- (-3.84312,4.99273 E-5) -- (-3.83206,5.96317 E-5) -- (-3.82100,7.05620 E-5) -- (-3.80994,8.28081 E-5) -- (-3.79888,9.64632 E-5) -- (-3.78783,0.000111623) -- (-3.77677,0.000128386) -- (-3.76571,0.000146853) -- (-3.75465,0.000167128) -- (-3.74359,0.000189316) -- (-3.73253,0.000213524) -- (-3.72147,0.000239861) -- (-3.71041,0.000268437) -- (-3.69935,0.000299363) -- (-3.68829,0.000332751) -- (-3.67723,0.000368712) -- (-3.66617,0.000407359) -- (-3.65511,0.000448802) -- (-3.64405,0.000493151) -- (-3.63299,0.000540515) -- (-3.62194,0.000591000) -- (-3.61088,0.000644710) -- (-3.59982,0.000701745) -- (-3.58876,0.000762201) -- (-3.57770,0.000826170) -- (-3.56664,0.000893739) -- (-3.55558,0.000964988) -- (-3.54452,0.00103999) -- (-3.53346,0.00111882) -- (-3.52240,0.00120152) -- (-3.51134,0.00128815) -- (-3.50028,0.00137874) -- (-3.48922,0.00147333) -- (-3.47816,0.00157192) -- (-3.46710,0.00167452) -- (-3.45604,0.00178113) -- (-3.44499,0.00189170) -- (-3.43393,0.00200620) -- (-3.42287,0.00212456) -- (-3.41181,0.00224672) -- (-3.40075,0.00237256) -- (-3.38969,0.00250197) -- (-3.37863,0.00263481) -- (-3.36757,0.00277091) -- (-3.35651,0.00291008) -- (-3.34545,0.00305211) -- (-3.33439,0.00319674) -- (-3.32333,0.00334372) -- (-3.31227,0.00349273) -- (-3.30121,0.00364344) -- (-3.29015,0.00379549) -- (-3.27910,0.00394848) -- (-3.26804,0.00410197) -- (-3.25698,0.00425548) -- (-3.24592,0.00440851) -- (-3.23486,0.00456051) -- (-3.22380,0.00471087) -- (-3.21274,0.00485897) -- (-3.20168,0.00500413) -- (-3.19062,0.00514562) -- (-3.17956,0.00528266) -- (-3.16850,0.00541445) -- (-3.15744,0.00554011) -- (-3.14638,0.00565872) -- (-3.13532,0.00576931) -- (-3.12426,0.00587085) -- (-3.11321,0.00596227) -- (-3.10215,0.00604242) -- (-3.09109,0.00611011) -- (-3.08003,0.00616410) -- (-3.06897,0.00620306) -- (-3.05791,0.00622564) -- (-3.04685,0.00623039) -- (-3.03579,0.00621583) -- (-3.02473,0.00618039) -- (-3.01367,0.00612247) -- (-3.00261,0.00604036) -- (-2.99155,0.00593232) -- (-2.98049,0.00579654) -- (-2.96943,0.00563113) -- (-2.95837,0.00543414) -- (-2.94732,0.00520356) -- (-2.93626,0.00493730) -- (-2.92520,0.00463321) -- (-2.91414,0.00428908) -- (-2.90308,0.00390262) -- (-2.89202,0.00347148) -- (-2.88096,0.00299323) -- (-2.86990,0.00246539) -- (-2.85884,0.00188541) -- (-2.84778,0.00125067) -- (-2.83672,0.000558485) -- (-2.82566,-0.000193889) -- (-2.81460,-0.00100926) -- (-2.80354,-0.00189051) -- (-2.79248,-0.00284056) -- (-2.78142,-0.00386241) -- (-2.77037,-0.00495911) -- (-2.75931,-0.00613377) -- (-2.74825,-0.00738954) -- (-2.73719,-0.00872965) -- (-2.72613,-0.0101574) -- (-2.71507,-0.0116760) -- (-2.70401,-0.0132889) -- (-2.69295,-0.0149994) -- (-2.68189,-0.0168111) -- (-2.67083,-0.0187274) -- (-2.65977,-0.0207519) -- (-2.64871,-0.0228880) -- (-2.63765,-0.0251394) -- (-2.62659,-0.0275097) -- (-2.61553,-0.0300024) -- (-2.60448,-0.0326214) -- (-2.59342,-0.0353702) -- (-2.58236,-0.0382524) -- (-2.57130,-0.0412719) -- (-2.56024,-0.0444323) -- (-2.54918,-0.0477373) -- (-2.53812,-0.0511905) -- (-2.52706,-0.0547957) -- (-2.51600,-0.0585565) -- (-2.50494,-0.0624766) -- (-2.49388,-0.0665595) -- (-2.48282,-0.0708089) -- (-2.47176,-0.0752283) -- (-2.46070,-0.0798213) -- (-2.44964,-0.0845912) -- (-2.43859,-0.0895417) -- (-2.42753,-0.0946759) -- (-2.41647,-0.0999973) -- (-2.40541,-0.105509) -- (-2.39435,-0.111214) -- (-2.38329,-0.117116) -- (-2.37223,-0.123218) -- (-2.36117,-0.129522) -- (-2.35011,-0.136031) -- (-2.33905,-0.142749) -- (-2.32799,-0.149677) -- (-2.31693,-0.156819) -- (-2.30587,-0.164176) -- (-2.29481,-0.171751) -- (-2.28375,-0.179546) -- (-2.27270,-0.187562) -- (-2.26164,-0.195803) -- (-2.25058,-0.204269) -- (-2.23952,-0.212962) -- (-2.22846,-0.221883) -- (-2.21740,-0.231034) -- (-2.20634,-0.240415) -- (-2.19528,-0.250028) -- (-2.18422,-0.259872) -- (-2.17316,-0.269950) -- (-2.16210,-0.280259) -- (-2.15104,-0.290802) -- (-2.13998,-0.301577) -- (-2.12892,-0.312584) -- (-2.11786,-0.323823) -- (-2.10680,-0.335293) -- (-2.09575,-0.346993) -- (-2.08469,-0.358920) -- (-2.07363,-0.371075) -- (-2.06257,-0.383455) -- (-2.05151,-0.396058) -- (-2.04045,-0.408882) -- (-2.02939,-0.421924) -- (-2.01833,-0.435181) -- (-2.00727,-0.448651) -- (-1.99621,-0.462330) -- (-1.98515,-0.476214) -- (-1.97409,-0.490300) -- (-1.96303,-0.504583) -- (-1.95197,-0.519059) -- (-1.94091,-0.533723) -- (-1.92986,-0.548571) -- (-1.91880,-0.563596) -- (-1.90774,-0.578794) -- (-1.89668,-0.594159) -- (-1.88562,-0.609684) -- (-1.87456,-0.625362) -- (-1.86350,-0.641189) -- (-1.85244,-0.657155) -- (-1.84138,-0.673254) -- (-1.83032,-0.689479) -- (-1.81926,-0.705822) -- (-1.80820,-0.722274) -- (-1.79714,-0.738827) -- (-1.78608,-0.755473) -- (-1.77502,-0.772202) -- (-1.76397,-0.789005) -- (-1.75291,-0.805874) -- (-1.74185,-0.822797) -- (-1.73079,-0.839766) -- (-1.71973,-0.856769) -- (-1.70867,-0.873797) -- (-1.69761,-0.890839) -- (-1.68655,-0.907884) -- (-1.67549,-0.924921) -- (-1.66443,-0.941938) -- (-1.65337,-0.958923) -- (-1.64231,-0.975866) -- (-1.63125,-0.992754) -- (-1.62019,-1.00958) -- (-1.60913,-1.02632) -- (-1.59808,-1.04297) -- (-1.58702,-1.05951) -- (-1.57596,-1.07594) -- (-1.56490,-1.09224) -- (-1.55384,-1.10839) -- (-1.54278,-1.12439) -- (-1.53172,-1.14022) -- (-1.52066,-1.15587) -- (-1.50960,-1.17132) -- (-1.49854,-1.18656) -- (-1.48748,-1.20158) -- (-1.47642,-1.21636) -- (-1.46536,-1.23089) -- (-1.45430,-1.24516) -- (-1.44324,-1.25915) -- (-1.43218,-1.27285) -- (-1.42113,-1.28624) -- (-1.41007,-1.29932) -- (-1.39901,-1.31207) -- (-1.38795,-1.32448) -- (-1.37689,-1.33652) -- (-1.36583,-1.34820) -- (-1.35477,-1.35950) -- (-1.34371,-1.37040) -- (-1.33265,-1.38089) -- (-1.32159,-1.39096) -- (-1.31053,-1.40060) -- (-1.29947,-1.40979) -- (-1.28841,-1.41853) -- (-1.27735,-1.42680) -- (-1.26629,-1.43459) -- (-1.25524,-1.44190) -- (-1.24418,-1.44870) -- (-1.23312,-1.45499) -- (-1.22206,-1.46077) -- (-1.21100,-1.46601) -- (-1.19994,-1.47072) -- (-1.18888,-1.47488) -- (-1.17782,-1.47848) -- (-1.16676,-1.48152) -- (-1.15570,-1.48399) -- (-1.14464,-1.48588) -- (-1.13358,-1.48718) -- (-1.12252,-1.48790) -- (-1.11146,-1.48802) -- (-1.10040,-1.48754) -- (-1.08935,-1.48645) -- (-1.07829,-1.48475) -- (-1.06723,-1.48245) -- (-1.05617,-1.47952) -- (-1.04511,-1.47598) -- (-1.03405,-1.47181) -- (-1.02299,-1.46703) -- (-1.01193,-1.46162) -- (-1.00087,-1.45559) -- (-0.989811,-1.44895) -- (-0.978752,-1.44168) -- (-0.967693,-1.43380) -- (-0.956633,-1.42530) -- (-0.945574,-1.41619) -- (-0.934515,-1.40647) -- (-0.923455,-1.39615) -- (-0.912396,-1.38524) -- (-0.901337,-1.37373) -- (-0.890277,-1.36164) -- (-0.879218,-1.34897) -- (-0.868158,-1.33573) -- (-0.857099,-1.32193) -- (-0.846040,-1.30759) -- (-0.834980,-1.29270) -- (-0.823921,-1.27728) -- (-0.812862,-1.26134) -- (-0.801802,-1.24490) -- (-0.790743,-1.22796) -- (-0.779684,-1.21054) -- (-0.768624,-1.19265) -- (-0.757565,-1.17431) -- (-0.746506,-1.15554) -- (-0.735446,-1.13634) -- (-0.724387,-1.11674) -- (-0.713328,-1.09676) -- (-0.702268,-1.07640) -- (-0.691209,-1.05569) -- (-0.680150,-1.03464) -- (-0.669090,-1.01329) -- (-0.658031,-0.991633) -- (-0.646972,-0.969707) -- (-0.635912,-0.947526) -- (-0.624853,-0.925112) -- (-0.613794,-0.902488) -- (-0.602734,-0.879673) -- (-0.591675,-0.856692) -- (-0.580616,-0.833565) -- (-0.569556,-0.810317) -- (-0.558497,-0.786970) -- (-0.547438,-0.763547) -- (-0.536378,-0.740073) -- (-0.525319,-0.716572) -- (-0.514259,-0.693067) -- (-0.503200,-0.669582) -- (-0.492141,-0.646143) -- (-0.481081,-0.622773) -- (-0.470022,-0.599498) -- (-0.458963,-0.576341) -- (-0.447903,-0.553328) -- (-0.436844,-0.530482) -- (-0.425785,-0.507829) -- (-0.414725,-0.485393) -- (-0.403666,-0.463198) -- (-0.392607,-0.441269) -- (-0.381547,-0.419629) -- (-0.370488,-0.398302) -- (-0.359429,-0.377312) -- (-0.348369,-0.356682) -- (-0.337310,-0.336435) -- (-0.326251,-0.316593) -- (-0.315191,-0.297180) -- (-0.304132,-0.278215) -- (-0.293073,-0.259722) -- (-0.282013,-0.241721) -- (-0.270954,-0.224232) -- (-0.259895,-0.207275) -- (-0.248835,-0.190870) -- (-0.237776,-0.175036) -- (-0.226717,-0.159789) -- (-0.215657,-0.145150) -- (-0.204598,-0.131133) -- (-0.193539,-0.117755) -- (-0.182479,-0.105032) -- (-0.171420,-0.0929784) -- (-0.160360,-0.0816084) -- (-0.149301,-0.0709351) -- (-0.138242,-0.0609710) -- (-0.127182,-0.0517277) -- (-0.116123,-0.0432160) -- (-0.105064,-0.0354460) -- (-0.0940044,-0.0284267) -- (-0.0829451,-0.0221663) -- (-0.0718857,-0.0166723) -- (-0.0608264,-0.0119511) -- (-0.0497670,-0.00800817) -- (-0.0387077,-0.00484826) -- (-0.0276484,-0.00247506) -- (-0.0165890,-0.000891372) -- (-0.00552967,-9.90608 E-5) -- (0.00552967,-9.90608 E-5) -- (0.0165890,-0.000891372) -- (0.0276484,-0.00247506) -- (0.0387077,-0.00484826) -- (0.0497670,-0.00800817) -- (0.0608264,-0.0119511) -- (0.0718857,-0.0166723) -- (0.0829451,-0.0221663) -- (0.0940044,-0.0284267) -- (0.105064,-0.0354460) -- (0.116123,-0.0432160) -- (0.127182,-0.0517277) -- (0.138242,-0.0609710) -- (0.149301,-0.0709351) -- (0.160360,-0.0816084) -- (0.171420,-0.0929784) -- (0.182479,-0.105032) -- (0.193539,-0.117755) -- (0.204598,-0.131133) -- (0.215657,-0.145150) -- (0.226717,-0.159789) -- (0.237776,-0.175036) -- (0.248835,-0.190870) -- (0.259895,-0.207275) -- (0.270954,-0.224232) -- (0.282013,-0.241721) -- (0.293073,-0.259722) -- (0.304132,-0.278215) -- (0.315191,-0.297180) -- (0.326251,-0.316593) -- (0.337310,-0.336435) -- (0.348369,-0.356682) -- (0.359429,-0.377312) -- (0.370488,-0.398302) -- (0.381547,-0.419629) -- (0.392607,-0.441269) -- (0.403666,-0.463198) -- (0.414725,-0.485393) -- (0.425785,-0.507829) -- (0.436844,-0.530482) -- (0.447903,-0.553328) -- (0.458963,-0.576341) -- (0.470022,-0.599498) -- (0.481081,-0.622773) -- (0.492141,-0.646143) -- (0.503200,-0.669582) -- (0.514259,-0.693067) -- (0.525319,-0.716572) -- (0.536378,-0.740073) -- (0.547438,-0.763547) -- (0.558497,-0.786970) -- (0.569556,-0.810317) -- (0.580616,-0.833565) -- (0.591675,-0.856692) -- (0.602734,-0.879673) -- (0.613794,-0.902488) -- (0.624853,-0.925112) -- (0.635912,-0.947526) -- (0.646972,-0.969707) -- (0.658031,-0.991633) -- (0.669090,-1.01329) -- (0.680150,-1.03464) -- (0.691209,-1.05569) -- (0.702268,-1.07640) -- (0.713328,-1.09676) -- (0.724387,-1.11674) -- (0.735446,-1.13634) -- (0.746506,-1.15554) -- (0.757565,-1.17431) -- (0.768624,-1.19265) -- (0.779684,-1.21054) -- (0.790743,-1.22796) -- (0.801802,-1.24490) -- (0.812862,-1.26134) -- (0.823921,-1.27728) -- (0.834980,-1.29270) -- (0.846040,-1.30759) -- (0.857099,-1.32193) -- (0.868158,-1.33573) -- (0.879218,-1.34897) -- (0.890277,-1.36164) -- (0.901337,-1.37373) -- (0.912396,-1.38524) -- (0.923455,-1.39615) -- (0.934515,-1.40647) -- (0.945574,-1.41619) -- (0.956633,-1.42530) -- (0.967693,-1.43380) -- (0.978752,-1.44168) -- (0.989811,-1.44895) -- (1.00087,-1.45559) -- (1.01193,-1.46162) -- (1.02299,-1.46703) -- (1.03405,-1.47181) -- (1.04511,-1.47598) -- (1.05617,-1.47952) -- (1.06723,-1.48245) -- (1.07829,-1.48475) -- (1.08935,-1.48645) -- (1.10040,-1.48754) -- (1.11146,-1.48802) -- (1.12252,-1.48790) -- (1.13358,-1.48718) -- (1.14464,-1.48588) -- (1.15570,-1.48399) -- (1.16676,-1.48152) -- (1.17782,-1.47848) -- (1.18888,-1.47488) -- (1.19994,-1.47072) -- (1.21100,-1.46601) -- (1.22206,-1.46077) -- (1.23312,-1.45499) -- (1.24418,-1.44870) -- (1.25524,-1.44190) -- (1.26629,-1.43459) -- (1.27735,-1.42680) -- (1.28841,-1.41853) -- (1.29947,-1.40979) -- (1.31053,-1.40060) -- (1.32159,-1.39096) -- (1.33265,-1.38089) -- (1.34371,-1.37040) -- (1.35477,-1.35950) -- (1.36583,-1.34820) -- (1.37689,-1.33652) -- (1.38795,-1.32448) -- (1.39901,-1.31207) -- (1.41007,-1.29932) -- (1.42113,-1.28624) -- (1.43218,-1.27285) -- (1.44324,-1.25915) -- (1.45430,-1.24516) -- (1.46536,-1.23089) -- (1.47642,-1.21636) -- (1.48748,-1.20158) -- (1.49854,-1.18656) -- (1.50960,-1.17132) -- (1.52066,-1.15587) -- (1.53172,-1.14022) -- (1.54278,-1.12439) -- (1.55384,-1.10839) -- (1.56490,-1.09224) -- (1.57596,-1.07594) -- (1.58702,-1.05951) -- (1.59808,-1.04297) -- (1.60913,-1.02632) -- (1.62019,-1.00958) -- (1.63125,-0.992754) -- (1.64231,-0.975866) -- (1.65337,-0.958923) -- (1.66443,-0.941938) -- (1.67549,-0.924921) -- (1.68655,-0.907884) -- (1.69761,-0.890839) -- (1.70867,-0.873797) -- (1.71973,-0.856769) -- (1.73079,-0.839766) -- (1.74185,-0.822797) -- (1.75291,-0.805874) -- (1.76397,-0.789005) -- (1.77502,-0.772202) -- (1.78608,-0.755473) -- (1.79714,-0.738827) -- (1.80820,-0.722274) -- (1.81926,-0.705822) -- (1.83032,-0.689479) -- (1.84138,-0.673254) -- (1.85244,-0.657155) -- (1.86350,-0.641189) -- (1.87456,-0.625362) -- (1.88562,-0.609684) -- (1.89668,-0.594159) -- (1.90774,-0.578794) -- (1.91880,-0.563596) -- (1.92986,-0.548571) -- (1.94091,-0.533723) -- (1.95197,-0.519059) -- (1.96303,-0.504583) -- (1.97409,-0.490300) -- (1.98515,-0.476214) -- (1.99621,-0.462330) -- (2.00727,-0.448651) -- (2.01833,-0.435181) -- (2.02939,-0.421924) -- (2.04045,-0.408882) -- (2.05151,-0.396058) -- (2.06257,-0.383455) -- (2.07363,-0.371075) -- (2.08469,-0.358920) -- (2.09575,-0.346993) -- (2.10680,-0.335293) -- (2.11786,-0.323823) -- (2.12892,-0.312584) -- (2.13998,-0.301577) -- (2.15104,-0.290802) -- (2.16210,-0.280259) -- (2.17316,-0.269950) -- (2.18422,-0.259872) -- (2.19528,-0.250028) -- (2.20634,-0.240415) -- (2.21740,-0.231034) -- (2.22846,-0.221883) -- (2.23952,-0.212962) -- (2.25058,-0.204269) -- (2.26164,-0.195803) -- (2.27270,-0.187562) -- (2.28375,-0.179546) -- (2.29481,-0.171751) -- (2.30587,-0.164176) -- (2.31693,-0.156819) -- (2.32799,-0.149677) -- (2.33905,-0.142749) -- (2.35011,-0.136031) -- (2.36117,-0.129522) -- (2.37223,-0.123218) -- (2.38329,-0.117116) -- (2.39435,-0.111214) -- (2.40541,-0.105509) -- (2.41647,-0.0999973) -- (2.42753,-0.0946759) -- (2.43859,-0.0895417) -- (2.44964,-0.0845912) -- (2.46070,-0.0798213) -- (2.47176,-0.0752283) -- (2.48282,-0.0708089) -- (2.49388,-0.0665595) -- (2.50494,-0.0624766) -- (2.51600,-0.0585565) -- (2.52706,-0.0547957) -- (2.53812,-0.0511905) -- (2.54918,-0.0477373) -- (2.56024,-0.0444323) -- (2.57130,-0.0412719) -- (2.58236,-0.0382524) -- (2.59342,-0.0353702) -- (2.60448,-0.0326214) -- (2.61553,-0.0300024) -- (2.62659,-0.0275097) -- (2.63765,-0.0251394) -- (2.64871,-0.0228880) -- (2.65977,-0.0207519) -- (2.67083,-0.0187274) -- (2.68189,-0.0168111) -- (2.69295,-0.0149994) -- (2.70401,-0.0132889) -- (2.71507,-0.0116760) -- (2.72613,-0.0101574) -- (2.73719,-0.00872965) -- (2.74825,-0.00738954) -- (2.75931,-0.00613377) -- (2.77037,-0.00495911) -- (2.78142,-0.00386241) -- (2.79248,-0.00284056) -- (2.80354,-0.00189051) -- (2.81460,-0.00100926) -- (2.82566,-0.000193889) -- (2.83672,0.000558485) -- (2.84778,0.00125067) -- (2.85884,0.00188541) -- (2.86990,0.00246539) -- (2.88096,0.00299323) -- (2.89202,0.00347148) -- (2.90308,0.00390262) -- (2.91414,0.00428908) -- (2.92520,0.00463321) -- (2.93626,0.00493730) -- (2.94732,0.00520356) -- (2.95837,0.00543414) -- (2.96943,0.00563113) -- (2.98049,0.00579654) -- (2.99155,0.00593232) -- (3.00261,0.00604036) -- (3.01367,0.00612247) -- (3.02473,0.00618039) -- (3.03579,0.00621583) -- (3.04685,0.00623039) -- (3.05791,0.00622564) -- (3.06897,0.00620306) -- (3.08003,0.00616410) -- (3.09109,0.00611011) -- (3.10215,0.00604242) -- (3.11321,0.00596227) -- (3.12426,0.00587085) -- (3.13532,0.00576931) -- (3.14638,0.00565872) -- (3.15744,0.00554011) -- (3.16850,0.00541445) -- (3.17956,0.00528266) -- (3.19062,0.00514562) -- (3.20168,0.00500413) -- (3.21274,0.00485897) -- (3.22380,0.00471087) -- (3.23486,0.00456051) -- (3.24592,0.00440851) -- (3.25698,0.00425548) -- (3.26804,0.00410197) -- (3.27910,0.00394848) -- (3.29015,0.00379549) -- (3.30121,0.00364344) -- (3.31227,0.00349273) -- (3.32333,0.00334372) -- (3.33439,0.00319674) -- (3.34545,0.00305211) -- (3.35651,0.00291008) -- (3.36757,0.00277091) -- (3.37863,0.00263481) -- (3.38969,0.00250197) -- (3.40075,0.00237256) -- (3.41181,0.00224672) -- (3.42287,0.00212456) -- (3.43393,0.00200620) -- (3.44499,0.00189170) -- (3.45604,0.00178113) -- (3.46710,0.00167452) -- (3.47816,0.00157192) -- (3.48922,0.00147333) -- (3.50028,0.00137874) -- (3.51134,0.00128815) -- (3.52240,0.00120152) -- (3.53346,0.00111882) -- (3.54452,0.00103999) -- (3.55558,0.000964988) -- (3.56664,0.000893739) -- (3.57770,0.000826170) -- (3.58876,0.000762201) -- (3.59982,0.000701745) -- (3.61088,0.000644710) -- (3.62194,0.000591000) -- (3.63299,0.000540515) -- (3.64405,0.000493151) -- (3.65511,0.000448802) -- (3.66617,0.000407359) -- (3.67723,0.000368712) -- (3.68829,0.000332751) -- (3.69935,0.000299363) -- (3.71041,0.000268437) -- (3.72147,0.000239861) -- (3.73253,0.000213524) -- (3.74359,0.000189316) -- (3.75465,0.000167128) -- (3.76571,0.000146853) -- (3.77677,0.000128386) -- (3.78783,0.000111623) -- (3.79888,9.64632 E-5) -- (3.80994,8.28081 E-5) -- (3.82100,7.05620 E-5) -- (3.83206,5.96317 E-5) -- (3.84312,4.99273 E-5) -- (3.85418,4.13618 E-5) -- (3.86524,3.38514 E-5) -- (3.87630,2.73154 E-5) -- (3.88736,2.16765 E-5) -- (3.89842,1.68608 E-5) -- (3.90948,1.27975 E-5) -- (3.92054,9.41950 E-6) -- (3.93160,6.66273 E-6) -- (3.94266,4.46663 E-6) -- (3.95372,2.77391 E-6) -- (3.96477,1.53054 E-6) -- (3.97583,6.85746 E-7) -- (3.98689,1.91891 E-7) -- (3.99795,4.45085 E-9) -- (4.00901,8.19284 E-8) -- (4.02007,3.85774 E-7) -- (4.03113,8.80302 E-7) -- (4.04219,1.53260 E-6) -- (4.05325,2.31242 E-6) -- (4.06431,3.19213 E-6) -- (4.07537,4.14652 E-6) -- (4.08643,5.15279 E-6) -- (4.09749,6.19040 E-6) -- (4.10855,7.24096 E-6) -- (4.11961,8.28814 E-6) -- (4.13066,9.31752 E-6) -- (4.14172,1.03166 E-5) -- (4.15278,1.12744 E-5) -- (4.16384,1.21818 E-5) -- (4.17490,1.30310 E-5) -- (4.18596,1.38158 E-5) -- (4.19702,1.45311 E-5) -- (4.20808,1.51731 E-5) -- (4.21914,1.57391 E-5) -- (4.23020,1.62273 E-5) -- (4.24126,1.66371 E-5) -- (4.25232,1.69684 E-5) -- (4.26338,1.72220 E-5) -- (4.27444,1.73994 E-5) -- (4.28550,1.75027 E-5) -- (4.29655,1.75343 E-5) -- (4.30761,1.74973 E-5) -- (4.31867,1.73950 E-5) -- (4.32973,1.72310 E-5) -- (4.34079,1.70094 E-5) -- (4.35185,1.67342 E-5) -- (4.36291,1.64097 E-5) -- (4.37397,1.60401 E-5) -- (4.38503,1.56301 E-5) -- (4.39609,1.51839 E-5) -- (4.40715,1.47059 E-5) -- (4.41821,1.42006 E-5) -- (4.42927,1.36721 E-5) -- (4.44033,1.31246 E-5) -- (4.45139,1.25621 E-5) -- (4.46245,1.19884 E-5) -- (4.47350,1.14073 E-5) -- (4.48456,1.08222 E-5) -- (4.49562,1.02365 E-5) -- (4.50668,9.65305 E-6) -- (4.51774,9.07489 E-6) -- (4.52880,8.50461 E-6) -- (4.53986,7.94462 E-6) -- (4.55092,7.39712 E-6) -- (4.56198,6.86406 E-6) -- (4.57304,6.34718 E-6) -- (4.58410,5.84801 E-6) -- (4.59516,5.36786 E-6) -- (4.60622,4.90781 E-6) -- (4.61728,4.46877 E-6) -- (4.62834,4.05145 E-6) -- (4.63939,3.65635 E-6) -- (4.65045,3.28384 E-6) -- (4.66151,2.93410 E-6) -- (4.67257,2.60715 E-6) -- (4.68363,2.30287 E-6) -- (4.69469,2.02104 E-6) -- (4.70575,1.76127 E-6) -- (4.71681,1.52309 E-6) -- (4.72787,1.30594 E-6) -- (4.73893,1.10914 E-6) -- (4.74999,9.31971 E-7) -- (4.76105,7.73617 E-7) -- (4.77211,6.33223 E-7) -- (4.78317,5.09885 E-7) -- (4.79423,4.02662 E-7) -- (4.80528,3.10587 E-7) -- (4.81634,2.32675 E-7) -- (4.82740,1.67929 E-7) -- (4.83846,1.15351 E-7) -- (4.84952,7.39469 E-8) -- (4.86058,4.27333 E-8) -- (4.87164,2.07443 E-8) -- (4.88270,7.03572 E-9) -- (4.89376,6.91131 E-10) -- (4.90482,8.25370 E-10) -- (4.91588,6.58863 E-9) -- (4.92694,1.71696 E-8) -- (4.93800,3.17983 E-8) -- (4.94906,4.97481 E-8) -- (4.96012,7.03376 E-8) -- (4.97117,9.29320 E-8) -- (4.98223,1.16944 E-7) -- (4.99329,1.41833 E-7) -- (5.00435,1.67107 E-7) -- (5.01541,1.92323 E-7) -- (5.02647,2.17083 E-7) -- (5.03753,2.41034 E-7) -- (5.04859,2.63870 E-7) -- (5.05965,2.85328 E-7) -- (5.07071,3.05185 E-7) -- (5.08177,3.23259 E-7) -- (5.09283,3.39406 E-7) -- (5.10389,3.53517 E-7) -- (5.11495,3.65516 E-7) -- (5.12601,3.75359 E-7) -- (5.13707,3.83029 E-7) -- (5.14812,3.88538 E-7) -- (5.15918,3.91917 E-7) -- (5.17024,3.93224 E-7) -- (5.18130,3.92531 E-7) -- (5.19236,3.89929 E-7) -- (5.20342,3.85522 E-7) -- (5.21448,3.79427 E-7) -- (5.22554,3.71769 E-7) -- (5.23660,3.62681 E-7) -- (5.24766,3.52304 E-7) -- (5.25872,3.40779 E-7) -- (5.26978,3.28252 E-7) -- (5.28084,3.14867 E-7) -- (5.29190,3.00769 E-7) -- (5.30296,2.86100 E-7) -- (5.31401,2.70996 E-7) -- (5.32507,2.55592 E-7) -- (5.33613,2.40014 E-7) -- (5.34719,2.24383 E-7) -- (5.35825,2.08814 E-7) -- (5.36931,1.93410 E-7) -- (5.38037,1.78270 E-7) -- (5.39143,1.63482 E-7) -- (5.40249,1.49126 E-7) -- (5.41355,1.35273 E-7) -- (5.42461,1.21983 E-7) -- (5.43567,1.09310 E-7) -- (5.44673,9.72978 E-8) -- (5.45779,8.59807 E-8) -- (5.46885,7.53857 E-8) -- (5.47990,6.55315 E-8) -- (5.49096,5.64291 E-8) -- (5.50202,4.80823 E-8) -- (5.51308,4.04882 E-8) -- (5.52414,3.36376 E-8) -- (5.53520,2.75159 E-8) -- (5.54626,2.21029 E-8) -- (5.55732,1.73743 E-8) -- (5.56838,1.33014 E-8) -- (5.57944,9.85213 E-9) -- (5.59050,6.99164 E-9) -- (5.60156,4.68256 E-9) -- (5.61262,2.88572 E-9) -- (5.62368,1.56060 E-9) -- (5.63474,6.65799 E-10) -- (5.64579,1.59528 E-10) -- (5.65685,0);
\end{tikzpicture}
\end{center}
\vskip 0.5cm
\begin{center}
\begin{tikzpicture}
\draw[black!50] (-5.65685,0) -- (5.65685,0);
\draw[black!50] (0,-2) -- (0,2);
\draw[black!50] (-0.1,2) -- (0.1,2);
\draw[black!50] (-0.1,-2) -- (0.1,-2);
\draw[black!50] (2,-0.1) -- (2,0.1);
\draw[black!50] (4,-0.1) -- (4,0.1);
\draw[black!50] (-2,-0.1) -- (-2,0.1);
\draw[black!50] (-4,-0.1) -- (-4,0.1);
\draw (-5.65685,0) -- (-5.64579,2.31528 E-5) -- (-5.63474,9.45676 E-5) -- (-5.62368,0.000216922) -- (-5.61262,0.000392522) -- (-5.60156,0.000623261) -- (-5.59050,0.000910594) -- (-5.57944,0.00125550) -- (-5.56838,0.00165845) -- (-5.55732,0.00211942) -- (-5.54626,0.00263781) -- (-5.53520,0.00321248) -- (-5.52414,0.00384175) -- (-5.51308,0.00452332) -- (-5.50202,0.00525438) -- (-5.49096,0.00603153) -- (-5.47990,0.00685083) -- (-5.46885,0.00770781) -- (-5.45779,0.00859751) -- (-5.44673,0.00951446) -- (-5.43567,0.0104528) -- (-5.42461,0.0114062) -- (-5.41355,0.0123681) -- (-5.40249,0.0133314) -- (-5.39143,0.0142891) -- (-5.38037,0.0152336) -- (-5.36931,0.0161575) -- (-5.35825,0.0170531) -- (-5.34719,0.0179129) -- (-5.33613,0.0187292) -- (-5.32507,0.0194948) -- (-5.31401,0.0202024) -- (-5.30296,0.0208452) -- (-5.29190,0.0214165) -- (-5.28084,0.0219103) -- (-5.26978,0.0223209) -- (-5.25872,0.0226434) -- (-5.24766,0.0228733) -- (-5.23660,0.0230069) -- (-5.22554,0.0230412) -- (-5.21448,0.0229742) -- (-5.20342,0.0228044) -- (-5.19236,0.0225315) -- (-5.18130,0.0221561) -- (-5.17024,0.0216796) -- (-5.15918,0.0211045) -- (-5.14812,0.0204344) -- (-5.13707,0.0196738) -- (-5.12601,0.0188281) -- (-5.11495,0.0179039) -- (-5.10389,0.0169088) -- (-5.09283,0.0158511) -- (-5.08177,0.0147402) -- (-5.07071,0.0135865) -- (-5.05965,0.0124011) -- (-5.04859,0.0111957) -- (-5.03753,0.00998305) -- (-5.02647,0.00877628) -- (-5.01541,0.00758918) -- (-5.00435,0.00643596) -- (-4.99329,0.00533120) -- (-4.98223,0.00428975) -- (-4.97117,0.00332663) -- (-4.96012,0.00245690) -- (-4.94906,0.00169556) -- (-4.93800,0.00105743) -- (-4.92694,0.000557055) -- (-4.91588,0.000208544) -- (-4.90482,2.54854 E-5) -- (-4.89376,2.08170 E-5) -- (-4.88270,0.000206709) -- (-4.87164,0.000594451) -- (-4.86058,0.00119434) -- (-4.84952,0.00201557) -- (-4.83846,0.00306615) -- (-4.82740,0.00435276) -- (-4.81634,0.00588074) -- (-4.80528,0.00765392) -- (-4.79423,0.00967461) -- (-4.78317,0.0119435) -- (-4.77211,0.0144597) -- (-4.76105,0.0172204) -- (-4.74999,0.0202214) -- (-4.73893,0.0234563) -- (-4.72787,0.0269174) -- (-4.71681,0.0305949) -- (-4.70575,0.0344773) -- (-4.69469,0.0385515) -- (-4.68363,0.0428026) -- (-4.67257,0.0472141) -- (-4.66151,0.0517679) -- (-4.65045,0.0564445) -- (-4.63939,0.0612227) -- (-4.62834,0.0660805) -- (-4.61728,0.0709941) -- (-4.60622,0.0759392) -- (-4.59516,0.0808901) -- (-4.58410,0.0858205) -- (-4.57304,0.0907036) -- (-4.56198,0.0955119) -- (-4.55092,0.100218) -- (-4.53986,0.104793) -- (-4.52880,0.109209) -- (-4.51774,0.113440) -- (-4.50668,0.117458) -- (-4.49562,0.121235) -- (-4.48456,0.124747) -- (-4.47350,0.127968) -- (-4.46245,0.130874) -- (-4.45139,0.133444) -- (-4.44033,0.135656) -- (-4.42927,0.137491) -- (-4.41821,0.138932) -- (-4.40715,0.139964) -- (-4.39609,0.140574) -- (-4.38503,0.140750) -- (-4.37397,0.140486) -- (-4.36291,0.139775) -- (-4.35185,0.138615) -- (-4.34079,0.137006) -- (-4.32973,0.134950) -- (-4.31867,0.132455) -- (-4.30761,0.129529) -- (-4.29655,0.126184) -- (-4.28550,0.122435) -- (-4.27444,0.118302) -- (-4.26338,0.113806) -- (-4.25232,0.108972) -- (-4.24126,0.103827) -- (-4.23020,0.0984023) -- (-4.21914,0.0927324) -- (-4.20808,0.0868537) -- (-4.19702,0.0808055) -- (-4.18596,0.0746301) -- (-4.17490,0.0683720) -- (-4.16384,0.0620782) -- (-4.15278,0.0557976) -- (-4.14172,0.0495813) -- (-4.13066,0.0434821) -- (-4.11961,0.0375542) -- (-4.10855,0.0318534) -- (-4.09749,0.0264364) -- (-4.08643,0.0213606) -- (-4.07537,0.0166843) -- (-4.06431,0.0124659) -- (-4.05325,0.00876384) -- (-4.04219,0.00563641) -- (-4.03113,0.00314136) -- (-4.02007,0.00133566) -- (-4.00901,0.000275192) -- (-3.99795,1.45026 E-5) -- (-3.98689,0.000606492) -- (-3.97583,0.00210215) -- (-3.96477,0.00455029) -- (-3.95372,0.00799726) -- (-3.94266,0.0124867) -- (-3.93160,0.0180593) -- (-3.92054,0.0247525) -- (-3.90948,0.0326004) -- (-3.89842,0.0416333) -- (-3.88736,0.0518777) -- (-3.87630,0.0633561) -- (-3.86524,0.0760866) -- (-3.85418,0.0900831) -- (-3.84312,0.105355) -- (-3.83206,0.121906) -- (-3.82100,0.139737) -- (-3.80994,0.158843) -- (-3.79888,0.179213) -- (-3.78783,0.200832) -- (-3.77677,0.223681) -- (-3.76571,0.247734) -- (-3.75465,0.272961) -- (-3.74359,0.299327) -- (-3.73253,0.326792) -- (-3.72147,0.355310) -- (-3.71041,0.384833) -- (-3.69935,0.415306) -- (-3.68829,0.446669) -- (-3.67723,0.478860) -- (-3.66617,0.511810) -- (-3.65511,0.545448) -- (-3.64405,0.579699) -- (-3.63299,0.614483) -- (-3.62194,0.649718) -- (-3.61088,0.685318) -- (-3.59982,0.721196) -- (-3.58876,0.757258) -- (-3.57770,0.793413) -- (-3.56664,0.829565) -- (-3.55558,0.865617) -- (-3.54452,0.901470) -- (-3.53346,0.937024) -- (-3.52240,0.972180) -- (-3.51134,1.00684) -- (-3.50028,1.04089) -- (-3.48922,1.07425) -- (-3.47816,1.10680) -- (-3.46710,1.13846) -- (-3.45604,1.16912) -- (-3.44499,1.19868) -- (-3.43393,1.22706) -- (-3.42287,1.25415) -- (-3.41181,1.27988) -- (-3.40075,1.30415) -- (-3.38969,1.32688) -- (-3.37863,1.34799) -- (-3.36757,1.36740) -- (-3.35651,1.38504) -- (-3.34545,1.40085) -- (-3.33439,1.41475) -- (-3.32333,1.42670) -- (-3.31227,1.43662) -- (-3.30121,1.44448) -- (-3.29015,1.45023) -- (-3.27910,1.45383) -- (-3.26804,1.45525) -- (-3.25698,1.45445) -- (-3.24592,1.45142) -- (-3.23486,1.44614) -- (-3.22380,1.43859) -- (-3.21274,1.42878) -- (-3.20168,1.41669) -- (-3.19062,1.40234) -- (-3.17956,1.38574) -- (-3.16850,1.36691) -- (-3.15744,1.34586) -- (-3.14638,1.32262) -- (-3.13532,1.29724) -- (-3.12426,1.26975) -- (-3.11321,1.24019) -- (-3.10215,1.20861) -- (-3.09109,1.17507) -- (-3.08003,1.13962) -- (-3.06897,1.10233) -- (-3.05791,1.06327) -- (-3.04685,1.02251) -- (-3.03579,0.980117) -- (-3.02473,0.936181) -- (-3.01367,0.890781) -- (-3.00261,0.844003) -- (-2.99155,0.795935) -- (-2.98049,0.746669) -- (-2.96943,0.696298) -- (-2.95837,0.644919) -- (-2.94732,0.592629) -- (-2.93626,0.539526) -- (-2.92520,0.485711) -- (-2.91414,0.431285) -- (-2.90308,0.376350) -- (-2.89202,0.321008) -- (-2.88096,0.265361) -- (-2.86990,0.209512) -- (-2.85884,0.153562) -- (-2.84778,0.0976123) -- (-2.83672,0.0417624) -- (-2.82566,-0.0138888) -- (-2.81460,-0.0692441) -- (-2.80354,-0.124208) -- (-2.79248,-0.178687) -- (-2.78142,-0.232589) -- (-2.77037,-0.285826) -- (-2.75931,-0.338311) -- (-2.74825,-0.389961) -- (-2.73719,-0.440694) -- (-2.72613,-0.490433) -- (-2.71507,-0.539104) -- (-2.70401,-0.586636) -- (-2.69295,-0.632960) -- (-2.68189,-0.678015) -- (-2.67083,-0.721738) -- (-2.65977,-0.764075) -- (-2.64871,-0.804973) -- (-2.63765,-0.844384) -- (-2.62659,-0.882263) -- (-2.61553,-0.918570) -- (-2.60448,-0.953269) -- (-2.59342,-0.986328) -- (-2.58236,-1.01772) -- (-2.57130,-1.04742) -- (-2.56024,-1.07541) -- (-2.54918,-1.10167) -- (-2.53812,-1.12619) -- (-2.52706,-1.14897) -- (-2.51600,-1.16999) -- (-2.50494,-1.18927) -- (-2.49388,-1.20680) -- (-2.48282,-1.22259) -- (-2.47176,-1.23665) -- (-2.46070,-1.24900) -- (-2.44964,-1.25965) -- (-2.43859,-1.26862) -- (-2.42753,-1.27595) -- (-2.41647,-1.28164) -- (-2.40541,-1.28574) -- (-2.39435,-1.28828) -- (-2.38329,-1.28929) -- (-2.37223,-1.28880) -- (-2.36117,-1.28686) -- (-2.35011,-1.28351) -- (-2.33905,-1.27879) -- (-2.32799,-1.27275) -- (-2.31693,-1.26542) -- (-2.30587,-1.25687) -- (-2.29481,-1.24713) -- (-2.28375,-1.23626) -- (-2.27270,-1.22430) -- (-2.26164,-1.21132) -- (-2.25058,-1.19736) -- (-2.23952,-1.18247) -- (-2.22846,-1.16671) -- (-2.21740,-1.15013) -- (-2.20634,-1.13278) -- (-2.19528,-1.11472) -- (-2.18422,-1.09600) -- (-2.17316,-1.07667) -- (-2.16210,-1.05679) -- (-2.15104,-1.03640) -- (-2.13998,-1.01555) -- (-2.12892,-0.994302) -- (-2.11786,-0.972697) -- (-2.10680,-0.950785) -- (-2.09575,-0.928612) -- (-2.08469,-0.906225) -- (-2.07363,-0.883667) -- (-2.06257,-0.860983) -- (-2.05151,-0.838214) -- (-2.04045,-0.815400) -- (-2.02939,-0.792581) -- (-2.01833,-0.769794) -- (-2.00727,-0.747075) -- (-1.99621,-0.724459) -- (-1.98515,-0.701979) -- (-1.97409,-0.679665) -- (-1.96303,-0.657548) -- (-1.95197,-0.635655) -- (-1.94091,-0.614013) -- (-1.92986,-0.592646) -- (-1.91880,-0.571578) -- (-1.90774,-0.550829) -- (-1.89668,-0.530421) -- (-1.88562,-0.510371) -- (-1.87456,-0.490695) -- (-1.86350,-0.471409) -- (-1.85244,-0.452526) -- (-1.84138,-0.434059) -- (-1.83032,-0.416018) -- (-1.81926,-0.398412) -- (-1.80820,-0.381250) -- (-1.79714,-0.364538) -- (-1.78608,-0.348281) -- (-1.77502,-0.332483) -- (-1.76397,-0.317147) -- (-1.75291,-0.302275) -- (-1.74185,-0.287867) -- (-1.73079,-0.273923) -- (-1.71973,-0.260442) -- (-1.70867,-0.247421) -- (-1.69761,-0.234858) -- (-1.68655,-0.222747) -- (-1.67549,-0.211084) -- (-1.66443,-0.199865) -- (-1.65337,-0.189082) -- (-1.64231,-0.178729) -- (-1.63125,-0.168798) -- (-1.62019,-0.159282) -- (-1.60913,-0.150171) -- (-1.59808,-0.141459) -- (-1.58702,-0.133134) -- (-1.57596,-0.125188) -- (-1.56490,-0.117612) -- (-1.55384,-0.110394) -- (-1.54278,-0.103525) -- (-1.53172,-0.0969950) -- (-1.52066,-0.0907926) -- (-1.50960,-0.0849076) -- (-1.49854,-0.0793294) -- (-1.48748,-0.0740473) -- (-1.47642,-0.0690507) -- (-1.46536,-0.0643290) -- (-1.45430,-0.0598715) -- (-1.44324,-0.0556679) -- (-1.43218,-0.0517079) -- (-1.42113,-0.0479812) -- (-1.41007,-0.0444777) -- (-1.39901,-0.0411876) -- (-1.38795,-0.0381011) -- (-1.37689,-0.0352089) -- (-1.36583,-0.0325016) -- (-1.35477,-0.0299701) -- (-1.34371,-0.0276056) -- (-1.33265,-0.0253997) -- (-1.32159,-0.0233439) -- (-1.31053,-0.0214303) -- (-1.29947,-0.0196510) -- (-1.28841,-0.0179985) -- (-1.27735,-0.0164656) -- (-1.26629,-0.0150453) -- (-1.25524,-0.0137310) -- (-1.24418,-0.0125161) -- (-1.23312,-0.0113946) -- (-1.22206,-0.0103605) -- (-1.21100,-0.00940819) -- (-1.19994,-0.00853237) -- (-1.18888,-0.00772790) -- (-1.17782,-0.00698995) -- (-1.16676,-0.00631389) -- (-1.15570,-0.00569538) -- (-1.14464,-0.00513027) -- (-1.13358,-0.00461467) -- (-1.12252,-0.00414489) -- (-1.11146,-0.00371748) -- (-1.10040,-0.00332916) -- (-1.08935,-0.00297689) -- (-1.07829,-0.00265778) -- (-1.06723,-0.00236915) -- (-1.05617,-0.00210849) -- (-1.04511,-0.00187345) -- (-1.03405,-0.00166185) -- (-1.02299,-0.00147167) -- (-1.01193,-0.00130101) -- (-1.00087,-0.00114812) -- (-0.989811,-0.00101139) -- (-0.978752,-0.000889319) -- (-0.967693,-0.000780528) -- (-0.956633,-0.000683747) -- (-0.945574,-0.000597807) -- (-0.934515,-0.000521636) -- (-0.923455,-0.000454252) -- (-0.912396,-0.000394757) -- (-0.901337,-0.000342332) -- (-0.890277,-0.000296230) -- (-0.879218,-0.000255772) -- (-0.868158,-0.000220343) -- (-0.857099,-0.000189383) -- (-0.846040,-0.000162390) -- (-0.834980,-0.000138908) -- (-0.823921,-0.000118527) -- (-0.812862,-0.000100880) -- (-0.801802,-8.56375 E-5) -- (-0.790743,-7.25041 E-5) -- (-0.779684,-6.12169 E-5) -- (-0.768624,-5.15418 E-5) -- (-0.757565,-4.32707 E-5) -- (-0.746506,-3.62192 E-5) -- (-0.735446,-3.02246 E-5) -- (-0.724387,-2.51432 E-5) -- (-0.713328,-2.08487 E-5) -- (-0.702268,-1.72304 E-5) -- (-0.691209,-1.41915 E-5) -- (-0.680150,-1.16474 E-5) -- (-0.669090,-9.52476 E-6) -- (-0.658031,-7.75980 E-6) -- (-0.646972,-6.29749 E-6) -- (-0.635912,-5.09038 E-6) -- (-0.624853,-4.09770 E-6) -- (-0.613794,-3.28456 E-6) -- (-0.602734,-2.62117 E-6) -- (-0.591675,-2.08223 E-6) -- (-0.580616,-1.64628 E-6) -- (-0.569556,-1.29523 E-6) -- (-0.558497,-1.01386 E-6) -- (-0.547438,-7.89416 E-7) -- (-0.536378,-6.11289 E-7) -- (-0.525319,-4.70655 E-7) -- (-0.514259,-3.60222 E-7) -- (-0.503200,-2.73995 E-7) -- (-0.492141,-2.07063 E-7) -- (-0.481081,-1.55426 E-7) -- (-0.470022,-1.15845 E-7) -- (-0.458963,-8.57080 E-8) -- (-0.447903,-6.29217 E-8) -- (-0.436844,-4.58197 E-8) -- (-0.425785,-3.30825 E-8) -- (-0.414725,-2.36728 E-8) -- (-0.403666,-1.67804 E-8) -- (-0.392607,-1.17768 E-8) -- (-0.381547,-8.17878 E-9) -- (-0.370488,-5.61718 E-9) -- (-0.359429,-3.81268 E-9) -- (-0.348369,-2.55570 E-9) -- (-0.337310,-1.69049 E-9) -- (-0.326251,-1.10244 E-9) -- (-0.315191,-7.08150 E-10) -- (-0.304132,-4.47562 E-10) -- (-0.293073,-2.77988 E-10) -- (-0.282013,-1.69461 E-10) -- (-0.270954,-1.01237 E-10) -- (-0.259895,0) -- (-0.248835,0) -- (-0.237776,0) -- (-0.226717,0) -- (-0.215657,0) -- (-0.204598,0) -- (-0.193539,0) -- (-0.182479,0) -- (-0.171420,0) -- (-0.160360,0) -- (-0.149301,0) -- (-0.138242,0) -- (-0.127182,0) -- (-0.116123,0) -- (-0.105064,0) -- (-0.0940044,0) -- (-0.0829451,0) -- (-0.0718857,0) -- (-0.0608264,0) -- (-0.0497670,0) -- (-0.0387077,0) -- (-0.0276484,0) -- (-0.0165890,0) -- (-0.00552967,0) -- (0.00552967,0) -- (0.0165890,0) -- (0.0276484,0) -- (0.0387077,0) -- (0.0497670,0) -- (0.0608264,0) -- (0.0718857,0) -- (0.0829451,0) -- (0.0940044,0) -- (0.105064,0) -- (0.116123,0) -- (0.127182,0) -- (0.138242,0) -- (0.149301,0) -- (0.160360,0) -- (0.171420,0) -- (0.182479,0) -- (0.193539,0) -- (0.204598,0) -- (0.215657,0) -- (0.226717,0) -- (0.237776,0) -- (0.248835,0) -- (0.259895,0) -- (0.270954,-1.01237 E-10) -- (0.282013,-1.69461 E-10) -- (0.293073,-2.77988 E-10) -- (0.304132,-4.47562 E-10) -- (0.315191,-7.08150 E-10) -- (0.326251,-1.10244 E-9) -- (0.337310,-1.69049 E-9) -- (0.348369,-2.55570 E-9) -- (0.359429,-3.81268 E-9) -- (0.370488,-5.61718 E-9) -- (0.381547,-8.17878 E-9) -- (0.392607,-1.17768 E-8) -- (0.403666,-1.67804 E-8) -- (0.414725,-2.36728 E-8) -- (0.425785,-3.30825 E-8) -- (0.436844,-4.58197 E-8) -- (0.447903,-6.29217 E-8) -- (0.458963,-8.57080 E-8) -- (0.470022,-1.15845 E-7) -- (0.481081,-1.55426 E-7) -- (0.492141,-2.07063 E-7) -- (0.503200,-2.73995 E-7) -- (0.514259,-3.60222 E-7) -- (0.525319,-4.70655 E-7) -- (0.536378,-6.11289 E-7) -- (0.547438,-7.89416 E-7) -- (0.558497,-1.01386 E-6) -- (0.569556,-1.29523 E-6) -- (0.580616,-1.64628 E-6) -- (0.591675,-2.08223 E-6) -- (0.602734,-2.62117 E-6) -- (0.613794,-3.28456 E-6) -- (0.624853,-4.09770 E-6) -- (0.635912,-5.09038 E-6) -- (0.646972,-6.29749 E-6) -- (0.658031,-7.75980 E-6) -- (0.669090,-9.52476 E-6) -- (0.680150,-1.16474 E-5) -- (0.691209,-1.41915 E-5) -- (0.702268,-1.72304 E-5) -- (0.713328,-2.08487 E-5) -- (0.724387,-2.51432 E-5) -- (0.735446,-3.02246 E-5) -- (0.746506,-3.62192 E-5) -- (0.757565,-4.32707 E-5) -- (0.768624,-5.15418 E-5) -- (0.779684,-6.12169 E-5) -- (0.790743,-7.25041 E-5) -- (0.801802,-8.56375 E-5) -- (0.812862,-0.000100880) -- (0.823921,-0.000118527) -- (0.834980,-0.000138908) -- (0.846040,-0.000162390) -- (0.857099,-0.000189383) -- (0.868158,-0.000220343) -- (0.879218,-0.000255772) -- (0.890277,-0.000296230) -- (0.901337,-0.000342332) -- (0.912396,-0.000394757) -- (0.923455,-0.000454252) -- (0.934515,-0.000521636) -- (0.945574,-0.000597807) -- (0.956633,-0.000683747) -- (0.967693,-0.000780528) -- (0.978752,-0.000889319) -- (0.989811,-0.00101139) -- (1.00087,-0.00114812) -- (1.01193,-0.00130101) -- (1.02299,-0.00147167) -- (1.03405,-0.00166185) -- (1.04511,-0.00187345) -- (1.05617,-0.00210849) -- (1.06723,-0.00236915) -- (1.07829,-0.00265778) -- (1.08935,-0.00297689) -- (1.10040,-0.00332916) -- (1.11146,-0.00371748) -- (1.12252,-0.00414489) -- (1.13358,-0.00461467) -- (1.14464,-0.00513027) -- (1.15570,-0.00569538) -- (1.16676,-0.00631389) -- (1.17782,-0.00698995) -- (1.18888,-0.00772790) -- (1.19994,-0.00853237) -- (1.21100,-0.00940819) -- (1.22206,-0.0103605) -- (1.23312,-0.0113946) -- (1.24418,-0.0125161) -- (1.25524,-0.0137310) -- (1.26629,-0.0150453) -- (1.27735,-0.0164656) -- (1.28841,-0.0179985) -- (1.29947,-0.0196510) -- (1.31053,-0.0214303) -- (1.32159,-0.0233439) -- (1.33265,-0.0253997) -- (1.34371,-0.0276056) -- (1.35477,-0.0299701) -- (1.36583,-0.0325016) -- (1.37689,-0.0352089) -- (1.38795,-0.0381011) -- (1.39901,-0.0411876) -- (1.41007,-0.0444777) -- (1.42113,-0.0479812) -- (1.43218,-0.0517079) -- (1.44324,-0.0556679) -- (1.45430,-0.0598715) -- (1.46536,-0.0643290) -- (1.47642,-0.0690507) -- (1.48748,-0.0740473) -- (1.49854,-0.0793294) -- (1.50960,-0.0849076) -- (1.52066,-0.0907926) -- (1.53172,-0.0969950) -- (1.54278,-0.103525) -- (1.55384,-0.110394) -- (1.56490,-0.117612) -- (1.57596,-0.125188) -- (1.58702,-0.133134) -- (1.59808,-0.141459) -- (1.60913,-0.150171) -- (1.62019,-0.159282) -- (1.63125,-0.168798) -- (1.64231,-0.178729) -- (1.65337,-0.189082) -- (1.66443,-0.199865) -- (1.67549,-0.211084) -- (1.68655,-0.222747) -- (1.69761,-0.234858) -- (1.70867,-0.247421) -- (1.71973,-0.260442) -- (1.73079,-0.273923) -- (1.74185,-0.287867) -- (1.75291,-0.302275) -- (1.76397,-0.317147) -- (1.77502,-0.332483) -- (1.78608,-0.348281) -- (1.79714,-0.364538) -- (1.80820,-0.381250) -- (1.81926,-0.398412) -- (1.83032,-0.416018) -- (1.84138,-0.434059) -- (1.85244,-0.452526) -- (1.86350,-0.471409) -- (1.87456,-0.490695) -- (1.88562,-0.510371) -- (1.89668,-0.530421) -- (1.90774,-0.550829) -- (1.91880,-0.571578) -- (1.92986,-0.592646) -- (1.94091,-0.614013) -- (1.95197,-0.635655) -- (1.96303,-0.657548) -- (1.97409,-0.679665) -- (1.98515,-0.701979) -- (1.99621,-0.724459) -- (2.00727,-0.747075) -- (2.01833,-0.769794) -- (2.02939,-0.792581) -- (2.04045,-0.815400) -- (2.05151,-0.838214) -- (2.06257,-0.860983) -- (2.07363,-0.883667) -- (2.08469,-0.906225) -- (2.09575,-0.928612) -- (2.10680,-0.950785) -- (2.11786,-0.972697) -- (2.12892,-0.994302) -- (2.13998,-1.01555) -- (2.15104,-1.03640) -- (2.16210,-1.05679) -- (2.17316,-1.07667) -- (2.18422,-1.09600) -- (2.19528,-1.11472) -- (2.20634,-1.13278) -- (2.21740,-1.15013) -- (2.22846,-1.16671) -- (2.23952,-1.18247) -- (2.25058,-1.19736) -- (2.26164,-1.21132) -- (2.27270,-1.22430) -- (2.28375,-1.23626) -- (2.29481,-1.24713) -- (2.30587,-1.25687) -- (2.31693,-1.26542) -- (2.32799,-1.27275) -- (2.33905,-1.27879) -- (2.35011,-1.28351) -- (2.36117,-1.28686) -- (2.37223,-1.28880) -- (2.38329,-1.28929) -- (2.39435,-1.28828) -- (2.40541,-1.28574) -- (2.41647,-1.28164) -- (2.42753,-1.27595) -- (2.43859,-1.26862) -- (2.44964,-1.25965) -- (2.46070,-1.24900) -- (2.47176,-1.23665) -- (2.48282,-1.22259) -- (2.49388,-1.20680) -- (2.50494,-1.18927) -- (2.51600,-1.16999) -- (2.52706,-1.14897) -- (2.53812,-1.12619) -- (2.54918,-1.10167) -- (2.56024,-1.07541) -- (2.57130,-1.04742) -- (2.58236,-1.01772) -- (2.59342,-0.986328) -- (2.60448,-0.953269) -- (2.61553,-0.918570) -- (2.62659,-0.882263) -- (2.63765,-0.844384) -- (2.64871,-0.804973) -- (2.65977,-0.764075) -- (2.67083,-0.721738) -- (2.68189,-0.678015) -- (2.69295,-0.632960) -- (2.70401,-0.586636) -- (2.71507,-0.539104) -- (2.72613,-0.490433) -- (2.73719,-0.440694) -- (2.74825,-0.389961) -- (2.75931,-0.338311) -- (2.77037,-0.285826) -- (2.78142,-0.232589) -- (2.79248,-0.178687) -- (2.80354,-0.124208) -- (2.81460,-0.0692441) -- (2.82566,-0.0138888) -- (2.83672,0.0417624) -- (2.84778,0.0976123) -- (2.85884,0.153562) -- (2.86990,0.209512) -- (2.88096,0.265361) -- (2.89202,0.321008) -- (2.90308,0.376350) -- (2.91414,0.431285) -- (2.92520,0.485711) -- (2.93626,0.539526) -- (2.94732,0.592629) -- (2.95837,0.644919) -- (2.96943,0.696298) -- (2.98049,0.746669) -- (2.99155,0.795935) -- (3.00261,0.844003) -- (3.01367,0.890781) -- (3.02473,0.936181) -- (3.03579,0.980117) -- (3.04685,1.02251) -- (3.05791,1.06327) -- (3.06897,1.10233) -- (3.08003,1.13962) -- (3.09109,1.17507) -- (3.10215,1.20861) -- (3.11321,1.24019) -- (3.12426,1.26975) -- (3.13532,1.29724) -- (3.14638,1.32262) -- (3.15744,1.34586) -- (3.16850,1.36691) -- (3.17956,1.38574) -- (3.19062,1.40234) -- (3.20168,1.41669) -- (3.21274,1.42878) -- (3.22380,1.43859) -- (3.23486,1.44614) -- (3.24592,1.45142) -- (3.25698,1.45445) -- (3.26804,1.45525) -- (3.27910,1.45383) -- (3.29015,1.45023) -- (3.30121,1.44448) -- (3.31227,1.43662) -- (3.32333,1.42670) -- (3.33439,1.41475) -- (3.34545,1.40085) -- (3.35651,1.38504) -- (3.36757,1.36740) -- (3.37863,1.34799) -- (3.38969,1.32688) -- (3.40075,1.30415) -- (3.41181,1.27988) -- (3.42287,1.25415) -- (3.43393,1.22706) -- (3.44499,1.19868) -- (3.45604,1.16912) -- (3.46710,1.13846) -- (3.47816,1.10680) -- (3.48922,1.07425) -- (3.50028,1.04089) -- (3.51134,1.00684) -- (3.52240,0.972180) -- (3.53346,0.937024) -- (3.54452,0.901470) -- (3.55558,0.865617) -- (3.56664,0.829565) -- (3.57770,0.793413) -- (3.58876,0.757258) -- (3.59982,0.721196) -- (3.61088,0.685318) -- (3.62194,0.649718) -- (3.63299,0.614483) -- (3.64405,0.579699) -- (3.65511,0.545448) -- (3.66617,0.511810) -- (3.67723,0.478860) -- (3.68829,0.446669) -- (3.69935,0.415306) -- (3.71041,0.384833) -- (3.72147,0.355310) -- (3.73253,0.326792) -- (3.74359,0.299327) -- (3.75465,0.272961) -- (3.76571,0.247734) -- (3.77677,0.223681) -- (3.78783,0.200832) -- (3.79888,0.179213) -- (3.80994,0.158843) -- (3.82100,0.139737) -- (3.83206,0.121906) -- (3.84312,0.105355) -- (3.85418,0.0900831) -- (3.86524,0.0760866) -- (3.87630,0.0633561) -- (3.88736,0.0518777) -- (3.89842,0.0416333) -- (3.90948,0.0326004) -- (3.92054,0.0247525) -- (3.93160,0.0180593) -- (3.94266,0.0124867) -- (3.95372,0.00799726) -- (3.96477,0.00455029) -- (3.97583,0.00210215) -- (3.98689,0.000606492) -- (3.99795,1.45026 E-5) -- (4.00901,0.000275192) -- (4.02007,0.00133566) -- (4.03113,0.00314136) -- (4.04219,0.00563641) -- (4.05325,0.00876384) -- (4.06431,0.0124659) -- (4.07537,0.0166843) -- (4.08643,0.0213606) -- (4.09749,0.0264364) -- (4.10855,0.0318534) -- (4.11961,0.0375542) -- (4.13066,0.0434821) -- (4.14172,0.0495813) -- (4.15278,0.0557976) -- (4.16384,0.0620782) -- (4.17490,0.0683720) -- (4.18596,0.0746301) -- (4.19702,0.0808055) -- (4.20808,0.0868537) -- (4.21914,0.0927324) -- (4.23020,0.0984023) -- (4.24126,0.103827) -- (4.25232,0.108972) -- (4.26338,0.113806) -- (4.27444,0.118302) -- (4.28550,0.122435) -- (4.29655,0.126184) -- (4.30761,0.129529) -- (4.31867,0.132455) -- (4.32973,0.134950) -- (4.34079,0.137006) -- (4.35185,0.138615) -- (4.36291,0.139775) -- (4.37397,0.140486) -- (4.38503,0.140750) -- (4.39609,0.140574) -- (4.40715,0.139964) -- (4.41821,0.138932) -- (4.42927,0.137491) -- (4.44033,0.135656) -- (4.45139,0.133444) -- (4.46245,0.130874) -- (4.47350,0.127968) -- (4.48456,0.124747) -- (4.49562,0.121235) -- (4.50668,0.117458) -- (4.51774,0.113440) -- (4.52880,0.109209) -- (4.53986,0.104793) -- (4.55092,0.100218) -- (4.56198,0.0955119) -- (4.57304,0.0907036) -- (4.58410,0.0858205) -- (4.59516,0.0808901) -- (4.60622,0.0759392) -- (4.61728,0.0709941) -- (4.62834,0.0660805) -- (4.63939,0.0612227) -- (4.65045,0.0564445) -- (4.66151,0.0517679) -- (4.67257,0.0472141) -- (4.68363,0.0428026) -- (4.69469,0.0385515) -- (4.70575,0.0344773) -- (4.71681,0.0305949) -- (4.72787,0.0269174) -- (4.73893,0.0234563) -- (4.74999,0.0202214) -- (4.76105,0.0172204) -- (4.77211,0.0144597) -- (4.78317,0.0119435) -- (4.79423,0.00967461) -- (4.80528,0.00765392) -- (4.81634,0.00588074) -- (4.82740,0.00435276) -- (4.83846,0.00306615) -- (4.84952,0.00201557) -- (4.86058,0.00119434) -- (4.87164,0.000594451) -- (4.88270,0.000206709) -- (4.89376,2.08170 E-5) -- (4.90482,2.54854 E-5) -- (4.91588,0.000208544) -- (4.92694,0.000557055) -- (4.93800,0.00105743) -- (4.94906,0.00169556) -- (4.96012,0.00245690) -- (4.97117,0.00332663) -- (4.98223,0.00428975) -- (4.99329,0.00533120) -- (5.00435,0.00643596) -- (5.01541,0.00758918) -- (5.02647,0.00877628) -- (5.03753,0.00998305) -- (5.04859,0.0111957) -- (5.05965,0.0124011) -- (5.07071,0.0135865) -- (5.08177,0.0147402) -- (5.09283,0.0158511) -- (5.10389,0.0169088) -- (5.11495,0.0179039) -- (5.12601,0.0188281) -- (5.13707,0.0196738) -- (5.14812,0.0204344) -- (5.15918,0.0211045) -- (5.17024,0.0216796) -- (5.18130,0.0221561) -- (5.19236,0.0225315) -- (5.20342,0.0228044) -- (5.21448,0.0229742) -- (5.22554,0.0230412) -- (5.23660,0.0230069) -- (5.24766,0.0228733) -- (5.25872,0.0226434) -- (5.26978,0.0223209) -- (5.28084,0.0219103) -- (5.29190,0.0214165) -- (5.30296,0.0208452) -- (5.31401,0.0202024) -- (5.32507,0.0194948) -- (5.33613,0.0187292) -- (5.34719,0.0179129) -- (5.35825,0.0170531) -- (5.36931,0.0161575) -- (5.38037,0.0152336) -- (5.39143,0.0142891) -- (5.40249,0.0133314) -- (5.41355,0.0123681) -- (5.42461,0.0114062) -- (5.43567,0.0104528) -- (5.44673,0.00951446) -- (5.45779,0.00859751) -- (5.46885,0.00770781) -- (5.47990,0.00685083) -- (5.49096,0.00603153) -- (5.50202,0.00525438) -- (5.51308,0.00452332) -- (5.52414,0.00384175) -- (5.53520,0.00321248) -- (5.54626,0.00263781) -- (5.55732,0.00211942) -- (5.56838,0.00165845) -- (5.57944,0.00125550) -- (5.59050,0.000910594) -- (5.60156,0.000623261) -- (5.61262,0.000392522) -- (5.62368,0.000216922) -- (5.63474,9.45676 E-5) -- (5.64579,2.31528 E-5) -- (5.65685,0);
\end{tikzpicture}
\end{center}
\caption{Two plots of the function $f$ from Theorem~\ref{thm_dimension12}.
The upper image is a cross section of the graph of $x \mapsto f(x)$ for $|x|^2 \le 8$;
note that this function decreases rapidly enough that the double roots
are nearly invisible.  The function in the lower image is instead proportional to
$x \mapsto |x|^{11} f(x)$.  This transformation distorts
the picture but clarifies the behavior, because $|x|^{11}$ is proportional to
the surface area of a sphere of radius $|x|$ in $\R^{12}$; thus, one-dimensional integrals of
the plotted function are proportional to integrals of $f$ in $\R^{12}$.}
\label{figure:dim12}
\end{figure}

The proof of Theorem~\ref{thm_dimension12} makes use of modular forms. The
lower bound $\bA_+(12) \ge \sqrt{2}$ follows from the existence of the
Eisenstein series $E_6$, while the upper bound $\bA_+(12) \le \sqrt{2}$ is
based on Viazovska's methods, which were developed to solve the sphere
packing problem in eight dimensions \cite{V} and twenty-four dimensions
\cite{CKMRV} (see also \cite{C2} for an exposition).  We prove both bounds
for $\bA_+(12)$ in Section~\ref{section:dimension12}.

The close relationship of this uncertainty principle with sphere packing may
seem surprising, given that Problem~\ref{+1problem} makes no reference to any
discrete structures. The connection is through the Euclidean linear
programming bound of Cohn and Elkies \cite{CE}, which converts a suitable
auxiliary function $f$ into an upper bound for the sphere packing density
$\Delta_d$ in $\R^d$. Suppose $f \colon \R^d \to \R$ is an integrable
function such that $\ft f$ is also integrable and real-valued (i.e., $f$ is
even), $f(0) = \ft f(0) = 1$, $\ft f \ge 0$ everywhere, and $f$ is eventually
nonpositive. Then the linear programming bound obtained from $f$ is the upper
bound
\begin{equation} \label{sphere_packing_bound}
\Delta_d \leq \vol\mathopen{}\big(B^d_{r(f)/2}\big)\mathclose{},
\end{equation}
where $B^d_R$ is the closed ball of radius $R$ about the origin in $\R^d$.
(Strictly speaking, the proof in \cite{CE} requires additional decay
hypotheses on $f$ and $\ft f$; see \cite[Theorem~3.3]{CZ} for a proof in the
generality of our statement here.) Optimizing this bound amounts to
minimizing $r(f)$.

Based on numerical evidence and analogies with other problems in coding
theory, Cohn and Elkies conjectured the existence of functions $f$ achieving
equality in \eqref{sphere_packing_bound} when $d \in \{2,8,24\}$, and they
proved it when $d=1$.  The case $d=2$ remains an open problem today, despite
the existence of elementary solutions of the two-dimensional sphere packing
problem by other means (see, for example, \cite{H}).  However, the case $d=8$
was proved fourteen years later in a breakthrough by Viazovska \cite{V}, and
the case $d=24$ was proved shortly thereafter based on her approach
\cite{CKMRV}. These papers solved the sphere packing problem in dimensions
$8$ and $24$.

The problem of optimizing the linear programming bound for $\Delta_d$ already
appears somewhat similar to Problem~\ref{+1problem}, but there is a deeper
analogy based on a problem studied by Cohn and Elkies in
\cite[Section~7]{CE}. Given an auxiliary function $f$ for the sphere packing
bound, let $g = \ft f - f$. Note that $g$ is not identically zero, because
otherwise $f$ and $\ft f$ would both have compact support (thanks to their
opposite signs outside radius $r(f)$), which would imply that $f = \ft f =
0$. Then $g$ satisfies the conditions of the following problem, with $r(g)
\le r(f)$:

\begin{problem}[$-1$ eigenfunction uncertainty principle]\label{-1problem}
Minimize $r(g)$ over all $g \colon \R^d \to \R$ such that
\begin{enumerate}
\item $g\in L^1(\R^d)\setminus\{0\}$ and $\ft g = -g$, and
\item $g(0)=0$ and $g$ is eventually nonnegative.
\end{enumerate}
\end{problem}

This problem has been solved for $d \in \{1, 8, 24\}$, as a consequence of
the sphere packing bounds mentioned above; the answers are $1$, $\sqrt{2}$,
and $2$, respectively. When $d=2$, it is conjectured that the optimal value
of $r(g)$ is $(4/3)^{1/4}$, but no proof is known.  No other closed forms
have been identified.

Cohn and Elkies conjectured \cite[Conjecture~7.2]{CE} that the minimal value
of $r(g)$ in Problem~\ref{-1problem} is exactly the same as that of $r(f)$ in
the linear programming bound, and that in fact an auxiliary function $f$ for
the linear programming bound can always be reconstructed from an optimal $g$
via $g = \ft f - f$. Nobody has proved that such an $f$ always exists, but
numerical evidence strongly supports this conjecture.

We can extend Problem~\ref{-1problem} to a broader uncertainty principle as
follows. Let $\A_-(d)$ denote the set of functions $f\colon\R^d\to \R$ such
that
\begin{enumerate}
\item $f\in L^1(\R^d)$, $\ft f\in L^1(\R^d)$, and $\ft f$ is real-valued
    (i.e., $f$ is even),
\item $f$ is eventually nonnegative while $\ft f(0)\leq 0$, and
\item $\ft f$ is eventually nonpositive while $f(0)\geq 0$.
\end{enumerate}
Let
\[
\bA_-(d) =  \inf_{f\in \A_-(d)\setminus\{0\}} \sqrt{r(f)r(\ft f\,)},
\]
and note that every function $g$ in Problem~\ref{-1problem} satisfies $r(g)
\ge \bA_-(d)$.

For completeness, we state our next theorem for both $\pm1$ cases, although
all the results in the following theorem were already proved for the $+1$
case by Gon\c calves, Oliveira e Silva, and Steinerberger in \cite{GOS}. Note
that we regard $\bA_{+1}$ and $\bA_{-1}$ as synonymous with $\bA_+$ and
$\bA_-$, respectively.

\begin{theorem}\label{thm_facts_+-1}
Let $s\in \{\pm 1\}$. Then there exist positive constants $c$ and $C$ such
that
\[
c \leq \frac{\bA_s(d)}{\sqrt{d}} \leq C
\]
for all $d$. Moreover, for each $d$ there exists a radial function $f\in
\A_s(d)\setminus\{0\}$ with $\ft f = s f$, $f(0)=0$, and
\[
r(f) = \bA_s(d).
\]
Furthermore, any such function must vanish at infinitely many radii greater
than $\bA_s(d)$.
\end{theorem}

In particular, $\bA_-(d) > 0$. Thus, we obtain a natural counterpart to the
uncertainty principle of Bourgain, Clozel, and Kahane, but with $f$ and $\ft
f$ having opposite signs, and with the optimal function coming from
Problem~\ref{-1problem}.  We can take $c = 1/\sqrt{2\pi e}$ and $C=1$.

This uncertainty principle places the linear programming bound in a broader
analytic context and gives a deeper significance to the auxiliary functions
that optimize this bound. Outside of a few exceptional dimensions, they do
not seem to come close to solving the sphere packing problem, but they
conjecturally achieve an optimal tradeoff between sign conditions in the
uncertainty principle.

Except for extremal functions for $\bA_+(1)$, our proof in
Section~\ref{root_property} and the proof in \cite{GOS} actually show that
any extremal function cannot be eventually positive; that is, it must vanish
on spheres with arbitrarily large radii, not just at infinitely many radii
greater than $\bA_s(d)$.  We strongly believe that this is the case for
$\bA_+(1)$ as well.

Problems~\ref{+1problem} and~\ref{-1problem} are closely related and behave
in complementary ways. We prove Theorem~\ref{thm_facts_+-1} by adapting the
techniques of \cite{GOS} to $-1$ eigenfunctions.  However, the analogy
between these problems is not perfect. For example, the equality
$\bA_+(12)=\bA_-(8)=\sqrt{2}$ suggests that perhaps $\bA_+(28) = \bA_-(24) =
2$, but that turns out to be false (see Section~\ref{section:numerics}).
Similarly, relatively simple explicit formulas show that $\bA_-(1)=1$, while
$\bA_+(1)$ remains a mystery.

In addition to its values in specific dimensions, the asymptotic behavior of
$\bA_s(d)$ as $d \to \infty$ is of substantial interest.  It was shown in
\cite{BCK} that
\[
0.2419\ldots = \frac{1}{\sqrt{2\pi e}}\leq \liminf_{d\to\infty} \frac{\bA_+(d)}{\sqrt{d}}
\leq \limsup_{d\to\infty} \frac{\bA_+(d)}{\sqrt{d}} \leq \frac{1}{\sqrt{2\pi}} = 0.3989\ldots.
\]
In Section~\ref{section:-1uncertainty}, we obtain the same lower bound for
the case of $\bA_-(d)$, and an improved upper bound of $0.3194\ldots$ for
that case based on \cite{CZ} (the exact value is complicated).

\begin{conjecture} \label{conj:limit}
The limits
\[
\lim_{d \to \infty} \frac{\bA_+(d)}{\sqrt{d}} \quad\text{and}\quad \lim_{d \to \infty} \frac{\bA_-(d)}{\sqrt{d}}
\]
exist and are equal.
\end{conjecture}

See Section~\ref{section:numerics} for the numerical evidence supporting this
conjecture.  We expect that the common value of these limits is strictly
between the bounds $0.2419\ldots$ and $0.3194\dotsc$, and perhaps not so far
from the latter.

In the remainder of the paper, we prove Theorem~\ref{thm_dimension12} in
Section~\ref{section:dimension12} and Theorem~\ref{thm_facts_+-1} in
Section~\ref{section:-1uncertainty}.  In Section~\ref{section:numerics} we
present numerical computations and conjectures, and we conclude in
Section~\ref{section:summation-formulas} with a construction of summation
formulas that validate our numerics and lend support to our general
conjectures about $\bA_s(d)$.

\section{The $+1$ eigenfunction uncertainty principle in dimension $12$}
\label{section:dimension12}

In this section, we prove Theorem~\ref{thm_dimension12}.

\subsection{Optimality} \label{optimality}

We begin by establishing that $\bA_+(12) \geq \sqrt{2}$. For this inequality,
we use a special Poisson-type summation formula for radial Schwartz functions
$f\colon\R^{12}\to \C$ based on the modular form $E_6$.  Converting a modular
form into such a formula is a standard technique; for completeness, we will
give a direct proof.

Consider the normalized Eisenstein series $E_6 \colon \h \to \C$, where $\h$
denotes the upper half-plane in $\C$ (see, for example, \cite[\S2]{Z}).  This
function has the Fourier expansion
\begin{equation} \label{eq:defE6}
E_6(z) = 1-\sum_{j\geq 1} c_j e^{2\pi i j z},
\end{equation}
where $c_j = 504\sigma_{5}(j)$ and $\sigma_5(j)$ is the sum of the fifth
powers of the divisors of $j$. In particular, $c_j >0$ for $j\geq 1$ and we
have the trivial bound $c_j\leq 504j^6$.  Because $E_6$ is a modular form of
weight $6$ for $\mathrm{SL}_2(\Z)$, it satisfies the identity
\begin{equation} \label{eq:E6identity}
E_6(z) = z^{-6}E_6(-1/z).
\end{equation}
This identity turns into a summation formula for a Gaussian $f \colon \R^{12}
\to \R$ defined by $f(x) = e^{-\pi \alpha|x|^2}$ with $\alpha>0$, or more
generally $\Re(\alpha) > 0$.  Specifically, if we set $z = i \alpha$, then
$f(x) = e^{\pi i z |x|^2}$ and $\ft f(\xi) = -z^{-6} e^{\pi i(-1/z)|\xi|^2}$,
from which it follows that $f(\sqrt{2j}) = e^{2\pi i j z}$ and $\ft
f(\sqrt{2j}) = -z^{-6} e^{2\pi i j (-1/z)}$, where we use $f(\sqrt{2j})$ to
denote the common value $f(x)$ with $|x|=\sqrt{2j}$.  Hence combining
\eqref{eq:defE6} and \eqref{eq:E6identity} yields
\[
f(0) - \sum_{j\geq 1} c_j f(\sqrt{2j}) = - \ft f(0) + \sum_{j\geq 1} c_j \ft f(\sqrt{2j}).
\]

The key to proving that $\bA_+(12) \geq \sqrt{2}$ is the following lemma,
which extends this summation formula to arbitrary radial Schwartz functions.

\begin{lemma}\label{E6_poisson_summation}
For all radial Schwartz functions $f\colon\R^{12}\to \C$,
\[
f(0) - \sum_{j\geq 1} c_j f(\sqrt{2j}) = - \ft f(0) + \sum_{j\geq 1} c_j \ft f(\sqrt{2j}).
\]
\end{lemma}

We follow the approach used to prove Theorem~1 in \cite[Section~6]{RV}.

\begin{proof}
Let $\Lambda\colon \Schw(\R^{12})\to \C$ be the functional
\[
\Lambda(f) = f(0) - \sum_{j\geq 1} c_j f(\sqrt{2j})  + \ft f(0) - \sum_{j\geq 1} c_j \ft f(\sqrt{2j})
\]
on the radial Schwartz space $\Schw(\R^{12})$.  As noted above,
$\Lambda(f)=0$ whenever $f(x) = e^{-\pi \alpha|x|^2}$ with $\Re(\alpha) > 0$.
Moreover, the bound $c_j =O(j^6)$ shows that $\Lambda$ is a continuous linear
functional in the topology of the Schwartz space. Thus, we need only prove
our desired identity for compactly supported, radial $C^\infty$ functions,
which are dense in $\Schw(\R^{12})$.

Write $f(x)=F(|x|^2)e^{-\pi |x|^2}$, where $F\colon\R \to \R$ is a smooth and
compactly supported function. Let $\ft F$ be the one-dimensional Fourier
transform of $F$, and note that $\ft F$ is also rapidly decreasing. By
Fourier inversion,
\[
f(x) = \int_\R \ft F(t)e^{-\pi (1-2i t) |x|^2}\,dt = \lim_{T\to \infty}
\int_{-T}^T \ft F(t)e^{-\pi (1-2i t) |x|^2}\,dt.
\]
The functions $x\mapsto\int_{-T}^T \ft F(t)e^{-\pi (1-2i t) |x|^2}\,dt$
belong to $\Schw(\R^{12})$ for each $T>0$ and converge to $f$ in the Schwartz
topology. Moreover,
\[
\Lambda\mathopen{}\left(x \mapsto \int_{-T}^T \ft F(t)e^{-\pi (1-2i t) |x|^2}\,dt\right)\mathclose{}
= \int_{-T}^T \ft F(t)\Lambda\mathopen{}\left(x \mapsto e^{-\pi (1-2i t) |x|^2}\right)\mathclose{}\,dt = 0,
\]
where the commutation is justified since the Riemann sums of the integral
converge to the integral in the topology of $\Schw(\R^{12})$. This finishes
the proof of the lemma.
\end{proof}

Noam Elkies has provided the following alternative proof of
Lemma~\ref{E6_poisson_summation} using Poisson summation.  Explicit
calculation shows that one can write the modular form $E_6$ in terms of theta
series of lattices and their duals as
\[
E_6 = -\frac{11}{10} \Theta_{D_{12}} + \frac{11}{20} \Theta_{D^*_{12}} - \frac{1}{20} \Theta_{L} + \frac{8}{5} \Theta_{L^*},
\]
where $L$ is the $D_{12}$ root lattice rescaled by a factor of $1/\sqrt{2}$.
Then the summation formula from Lemma~\ref{E6_poisson_summation} becomes a
linear combination of the Poisson summation formulas for the lattices
$D_{12}$, $D_{12}^*$, $L$, and $L^*$, which implies that it holds for all
radial Schwartz functions.  This argument shows that
Lemma~\ref{E6_poisson_summation} is closely related to Poisson summation,
while the proof we gave above applies directly to other modular forms as well
as $E_6$.

\begin{lemma} \label{lemma:vanishing}
Let $f \in \A_+(12)$.  If both $r(f)$ and $r(\ft f\,)$ are at most
$\sqrt{2}$, then $f(x) = \ft f(x) = 0$ whenever $|x| = \sqrt{2j}$ with $j$ a
nonnegative integer.
\end{lemma}

\begin{proof}
Without loss of generality, we can assume $f$ is a radial function;
otherwise, we simply average its rotations about the origin.  (If the
averaged function vanishes at radius $\sqrt{2j}$, then so does $f$ because
$r(f) \le \sqrt{2}$, and the same holds for $\ft f$.)

If $f$ is a radial Schwartz function, then Lemma~\ref{E6_poisson_summation}
implies that
\[
f(0) + \ft f(0)  = \sum_{j\geq 1} c_j f(\sqrt{2j})  + \sum_{j\geq 1} c_j \ft f(\sqrt{2j}),
\]
and the conclusion follows from the inequalities $f(0) \le 0$, $\ft f(0) \le
0$, $f(\sqrt{2j}) \ge 0$, $\ft f(\sqrt{2j}) \ge 0$, and $c_j > 0$ for $j \ge
1$.

For general $f$, we can apply a standard mollification argument. Let
$\varphi\colon \R^d \to \R$ be a nonnegative, radial $C^\infty$ function
supported in the unit ball $B^{d}_1$ with $\ft \varphi \geq 0$ and $\ft
\varphi(0)=1$, so that the functions $\varphi_\varepsilon$ defined for
$\varepsilon>0$ by $\varphi_\varepsilon(x) =
\varepsilon^{-d}\varphi(x/\varepsilon)$ form an approximate identity.

Now let $f_\varepsilon = (f * \varphi_\varepsilon) \ft\varphi_\varepsilon$.
Because $f$ and $\ft f$ are continuous functions that vanish at infinity,
$f_\varepsilon \to f$ and $\ft f_\varepsilon \to \ft f$ uniformly on $\R^d$
as $\varepsilon \to 0$. Since $\supp(\varphi_\varepsilon) \subseteq
B^{d}_\varepsilon$, we obtain the inequality $f_\varepsilon(x) \geq 0$
whenever $|x|\geq r(f) + \varepsilon$. Similarly $\ft{f_\varepsilon} = (\ft f
\ \ft\varphi_\varepsilon)*\varphi_\varepsilon$, which implies that $\ft
f_\varepsilon(x)\geq 0$ whenever $|x|\geq r(\ft f\,)+\varepsilon$.
Furthermore, $f_\varepsilon$ is a Schwartz function.  To see why, note that
$\ft\varphi_\varepsilon$ is a Schwartz function, while $f *
\varphi_\varepsilon$ is smooth and all its derivatives are bounded.

Now that we have Schwartz functions approximating $f$, we again apply
Lemma~\ref{E6_poisson_summation} to obtain
\[
f_\varepsilon(0) + \ft f_\varepsilon(0)  =
\sum_{j\geq 1} c_j f_\varepsilon(\sqrt{2j})
+ \sum_{j\geq 1} c_j \ft f_\varepsilon(\sqrt{2j}).
\]
To derive information from this identity, we combine the limits
$f_\varepsilon(\sqrt{2j}) \to f(\sqrt{2j})$ and $\ft f_\varepsilon(\sqrt{2j})
\to \ft f(\sqrt{2j})$ for $j \ge 0$, the inequalities $f(0) \le 0$, $\ft f(0)
\le 0$, $f(\sqrt{2}) \ge 0$, and $\ft f(\sqrt{2}) \ge 0$, and the
inequalities $f_\varepsilon(\sqrt{2j}) \ge 0$  and $\ft
f_\varepsilon(\sqrt{2j}) \ge 0$ for $j \ge 2$ (when $\varepsilon <
2-\sqrt{2}$).  We conclude that $f(\sqrt{2j}) = \ft f(\sqrt{2j}) = 0$ for $j
\ge 0$, as desired.
\end{proof}

We will now apply this lemma to prove the lower bound $\bA_+(12) \ge
\sqrt{2}$.

\begin{lemma}
Suppose $f \in \A_+(12)$.  If $r(f)r(\ft f\,) < 2$, then $f$ vanishes
identically.
\end{lemma}

\begin{proof}
By rescaling the input to $f$, we can assume without loss of generality that
$r(f)$ and $r(\ft f\,)$ are both less than $\sqrt{2}$.  Now we apply
Lemma~\ref{lemma:vanishing} to a rescaled version of $f$.  Choose $\lambda>0$
and let $g(x) = f(\lambda x)$.  Then $\ft g(\xi) = \lambda^{-12} \ft
f(\xi/\lambda)$, and it follows that $g \in \A_+(12)$. Moreover, if $\lambda$
is close enough to $1$, then $r(g)$ and $r(\ft g)$ are both less than
$\sqrt{2}$.

By Lemma~\ref{lemma:vanishing}, if $\lambda$ is sufficiently close to $1$,
then $g(x) = 0$ whenever $|x|=\sqrt{2j}$ with $j \ge 1$.  Thus there exists
some $\lambda_0>1$ such that $f(x)=0$ whenever
$|x|\in(\sqrt{2j}/\lambda_0,\sqrt{2j}\lambda_0)$ and $j\geq 1$, and the same
holds for $\ft f$.  The union of these intervals covers the entire half-line
$[R,\infty)$ for some $R>0$, because
\[
\lim_{j \to \infty} \frac{\sqrt{2j+2}}{\sqrt{2j}} = 1.
\]
In other words, $f$ and $\ft f$ both have compact support, which implies that
$f=0$.
\end{proof}

Exactly the same technique applies to any dimension and sign:

\begin{proposition} \label{prop:lowerbound}
Let $s \in\{\pm 1\}$, $0 < \rho_0 < \rho_1 < \dotsb$ with
\[
\lim_{j \to \infty} \frac{\rho_{j+1}}{\rho_j} = 1,
\]
and $c_j > 0$ for $j \ge 0$.  If every radial Schwartz function $f \colon
\R^d \to \R$ satisfies the summation formula
\begin{equation} \label{eq:general-summation}
f(0) + s\ft f(0) = s\sum_{j \ge 0} c_j f(\rho_j) + \sum_{j \ge 0} c_j \ft f(\rho_j),
\end{equation}
then $\bA_s(d) \ge \rho_0$.
\end{proposition}

For example, for $k\ge 2$, the summation formula coming from the Eisenstein
series $E_{2k}$ proves that $\bA_{(-1)^{k-1}}(4k) \ge \sqrt{2}$.  This lower
bound is sharp for $k=2$ and $k=3$, but it is not even true for $k=1$,
because $E_2$ is merely a quasimodular form.

The summation formula \eqref{eq:general-summation} automatically holds when
$\ft f = - s f$. Thus, it is equivalent to the assertion that
\begin{equation} \label{eq:general-summation2}
f(0) = s \sum_{j \ge 0} c_j f(\rho_j)
\end{equation}
holds whenever $\ft f = s f$.

\begin{conjecture} \label{conjecture:summation}
For each $s = \pm 1$ and $d \ge 1$ except perhaps $(s,d)=(1,1)$, there is a
summation formula that proves a sharp lower bound for $\bA_s(d)$ via
Proposition~\ref{prop:lowerbound}.
\end{conjecture}

In the case $s=-1$, this conjecture is analogous to \cite[Conjecture~4.2]{C}.
It holds in every case in which $\bA_s(d)$ is known exactly: the summation
formulas that establish sharp lower bounds for $\bA_{-}(1)$, $\bA_{-}(8)$,
and $\bA_{-}(24)$ are Poisson summation over the $\Z$, $E_8$, and Leech
lattices, respectively, while the $\bA_{+}(12)$ case is
Lemma~\ref{E6_poisson_summation}.  The conjectured value of $\bA_{-}(2)$
corresponds to Poisson summation over the isodual scaling of the $A_2$
lattice. Conjecture~\ref{conjecture:summation} is not known to hold in any
other case, nor can we guess what the summation formula should be, but the
numerical and theoretical evidence in favor of this conjecture is compelling
(see Sections~\ref{section:numerics} and~\ref{section:summation-formulas}).
In particular, in most cases we can compute the constants $c_j$ and $\rho_j$
in these conjectural summation formulas to high precision.

The coefficients $c_j$ are integers in the five exact cases listed above, but
integral coefficients seem to be rare, and it is plausible that no more such
cases exist.  One interesting example is the (conjectural) summation formula
that yields $\bA_{+}(28)$. It is natural to guess that $\bA_{+}(28)=2$, in
accordance with $\bA_{+}(12)=\bA_{-}(8)$ and $\bA_{-}(24)=2$, but in fact
$\bA_{+}(28) < 1.98540693489105$, and we conjecture that $\bA_{+}(28) =
1.985406934891049\ldots.$  (See Section~\ref{section:numerics} for a
discussion of our numerical methods.) In Table~\ref{table:summation-formula},
we approximate a conjectural summation formula that would establish this
equality, which we computed using the techniques of
Section~\ref{section:summation-formulas}.  We are unable to describe the
numbers $\rho_j$ and $c_j$ in the summation formula exactly, but we believe
that $\rho_j = \sqrt{2j+4+o(1)}$ as $j \to \infty$ (see
Conjecture~\ref{conjecture:28}) and $c_j = (24+o(1)) \sigma_{13}(j+2)$. The
latter equation says that $-c_j$ is asymptotic to the coefficient of
$e^{(2j+4)\pi i z}$ in the Fourier expansion
\begin{align*}
E_{14}(z) = 1 & - 24e^{2\pi i z} - 196632e^{4\pi i z} - 38263776e^{6\pi i z}
- 1610809368e^{8\pi i z}\\
& \phantom{} - 29296875024e^{10\pi i z} - 313495116768e^{12\pi i z}
- 2325336249792e^{14\pi i z}\\
& \phantom{} - 13195750342680e^{16\pi i z} - \cdots
\end{align*}
of the Eisenstein series $E_{14}$, and indeed these coefficients are close to
those in the table.  Note that the difference between the role of $E_{14}$
here and that of $E_6$ when $d=12$ is that the summation formula for $d=28$
suppresses the $-24 e^{2\pi i z}$ term in $E_{14}(z)$ at the cost of
perturbing all the remaining numbers.

\begin{table}
\caption{Summation formula that would prove $\bA_{+}(28) \ge 1.985406934891049\ldots.$
We conjecture that there exists a formula of the form \eqref{eq:general-summation} in $\R^{28}$
that agrees with all the digits listed in this table and proves a sharp lower bound for $\bA_{+}(28)$.}
\label{table:summation-formula}
\begin{tabular}{ccc}
\toprule
$j$ & $\rho_j\phantom{\ldots}$ & $c_j\phantom{\ldots}$\\
\midrule
$0$ & $1.985406934891049\ldots$ & $173693.2739265496\ldots$\\
$1$ & $2.448204775489784\ldots$ & $38022505.25862595\ldots$\\
$2$ & $2.828451453989980\ldots$ & $1612404204.870089\ldots$\\
$3$ & $3.162301096885930\ldots$ & $29295881893.82392\ldots$\\
$4$ & $3.464102777388629\ldots$ & $313503500519.3102\ldots$\\
$5$ & $3.741654846843136\ldots$ & $2325238355388.562\ldots$\\
$6$ & $3.999999847797149\ldots$ & $13196060863066.90\ldots$\\
$\vdots$ & $\vdots\phantom{\ldots}$ & $\vdots\phantom{\ldots}$\\
\bottomrule
\end{tabular}
\end{table}

\subsection{Theta series and an extremal function in dimension $12$}\label{construction_section}

To prove the upper bound $\bA_+(12) \le \sqrt{2}$, we will construct an
explicit function $f \in \A_+(12)$ satisfying $\ft f = f$, $f(0)=0$, and
$r(f) = \sqrt{2}$.  To do so, we will use a remarkable integral transform
discovered by Viazovska that turns modular forms into radial eigenfunctions
of the Fourier transform. See \cite{Z} for background on modular forms, and
\cite{V,CKMRV,CKMRV2} for other applications of this transform.

Viazovska's method can be summarized by the following proposition, which is
implicit in \cite{V} but was stated there only for a specific modular form
with $d=8$ (and similarly for $d=24$ in \cite{CKMRV}). We omit the proof,
because it closely follows the same approach as \cite[Propositions~5
and~6]{V} and \cite[Lemma~3.1]{CKMRV}.  All that needs to be checked is the
dependence on the dimension $d$.

\begin{proposition} \label{prop:psi}
Let $d$ be a positive multiple of $4$, and let $\psi$ be a weakly holomorphic
modular form of weight $2-d/2$ for $\Gamma(2)$ such that
\[
z^{d/2-2}\psi(-1/z) + \psi(z+1) = \psi(z)
\]
for all $z$ in the upper half-plane, $t^{d/2-2}\psi(i/t) \to 0$ as $t \to
\infty$, and $|\psi(it)| = O\big(e^{K\pi t}\big)$ as $t \to \infty$ for some
constant $K$. Define a radial function $f \colon \R^d \to \R$ by
\begin{equation}
\label{eq:f-def}
\begin{split}
f(x) & =
\frac{i}{4}\int_{-1}^i \psi(z+1) e^{\pi i |x|^2 z} \, dz + \frac{i}{4}\int_{1}^i \psi(z-1) e^{\pi i |x|^2 z} \, dz\\
& \quad \phantom{}  - \frac{i}{2}\int_{0}^i \psi(z) e^{\pi i |x|^2 z} \, dz
 - \frac{i}{2} \int_i^{i\infty} z^{d/2-2}\psi(-1/z) e^{\pi i |x|^2 z} \, dz.
\end{split}
\end{equation}
Then $f$ is a Schwartz function and an eigenfunction of the Fourier transform
with eigenvalue $(-1)^{1+d/4}$.  Furthermore,
\[
f(x) = \sin\mathopen{}\big(\pi |x|^2/2\big)^2\mathclose{} \int_0^\infty \psi(it) e^{-\pi |x|^2 t}\, dt
\]
whenever $|x|^2 > K$.
\end{proposition}

Viazovska in fact developed two such techniques, one for each eigenvalue, and
both are used in the sphere packing papers \cite{V,CKMRV}. We will not need
the other technique, which yields eigenvalue $(-1)^{d/4}$ instead of
$(-1)^{1+d/4}$ and uses a weakly holomorphic quasimodular form of weight
$4-d/2$ and depth $2$ for $\mathrm{SL}_2(\Z)$.

When applying Proposition~\ref{prop:psi}, we will use the notation
\begin{flalign*}
&& \Theta_{00}(z) & = \sum_{n \in \Z} e^{\pi i n^2 z}, &&\\
&& \Theta_{01}(z) & = \sum_{n \in \Z} (-1)^n e^{\pi i n^2 z}, &&\\
&\text{and}\hidewidth\\
&& \Theta_{10}(z) & = \sum_{n \in \Z} e^{\pi i (n+1/2)^2 z} &&
\end{flalign*}
for theta functions from \cite{V,CKMRV}.  Their fourth powers
$\Theta_{00}^4$, $\Theta_{01}^4$, and $\Theta_{10}^4$ are modular forms of
weight $2$ for $\Gamma(2)$, which satisfy the Jacobi identity $\Theta_{00}^4
= \Theta_{01}^4 + \Theta_{10}^4$ and the transformation laws
\begin{alignat*}{3}
\Theta_{00}(z+1)^4 &= \Theta_{01}(z)^4, &\qquad z^{-2}\Theta_{00}(-1/z)^4 &= -\Theta_{00}(z)^4,\\
\Theta_{01}(z+1)^4 &= \Theta_{00}(z)^4, &\qquad z^{-2}\Theta_{01}(-1/z)^4 &= -\Theta_{10}(z)^4,\\
\Theta_{10}(z+1)^4 &= -\Theta_{10}(z)^4, &\qquad z^{-2}\Theta_{10} (-1/z)^4 &= -\Theta_{01}(z)^4
\end{alignat*}
under the action of $\mathrm{SL}_2(\Z)$. We will also use the modular form
$\Delta$, defined by
\[
\Delta(z) = e^{2\pi i z} \prod_{n=1}^\infty (1-e^{2\pi i n z})^{24}.
\]
It is a modular form of weight $12$ for the group $\mathrm{SL}_2(\Z)$, which
contains $\Gamma(2)$; thus $\Delta(z+1) = \Delta(z)$ and $z^{-12}
\Delta(-1/z) = \Delta(z)$.

Using these ingredients, we will now construct a suitable modular form for
use in Proposition~\ref{prop:psi}, to prove Theorem~\ref{thm_dimension12}.
Let
\begin{equation}
\label{eq:psi-def}
\psi = \frac{\big(\Theta_{00}^4 + \Theta_{10}^4\big)\Theta_{01}^{12}}{\Delta}.
\end{equation}
(We discuss the motivation for this definition at the end of this section.)
Then $\psi$ is a weakly holomorphic modular form of weight $4\cdot2 - 12
=-4$, and the identity
\[
z^{4}\psi(-1/z) + \psi(z+1) = \psi(z)
\]
can be checked using the formulas listed above. (Note that $\psi$ is weakly
holomorphic because the product formula shows that $\Delta$ does not vanish
in the upper half-plane.)

Using the definitions for $\Theta_{00}$, $\Theta_{01}$, $\Theta_{10}$, and
$\Delta$ given above, we can compute the Fourier series
\begin{equation}
\label{eq:fourier}
\psi(z) = e^{-2\pi i z} - 264 + 4096 e^{\pi i z} - 36828 e^{2\pi i z} + 245760 e^{3\pi i z} + \cdots.
\end{equation}
This series is absolutely convergent in the upper half-plane, and thus
$|\psi(it)| = O\big(e^{2\pi t}\big)$ as $t \to \infty$.  Using the
transformation laws again, we find that
\begin{align*}
z^4 \psi(-1/z) &= \frac{\big(\Theta_{00}(z)^4 + \Theta_{01}(z)^4\big)\Theta_{10}(z)^{12}}{\Delta(z)}\\
&= 8192 e^{\pi i z} + 491520 e^{3\pi i z} + 12828672 e^{5 \pi i z} + \cdots.
\end{align*}
In particular, $|t^4 \psi(i/t)| = O\big(e^{-\pi t}\big)$ as $t \to \infty$.

Thus, $\psi$ satisfies the hypotheses of Proposition~\ref{prop:psi} with
$d=12$ and $K=2$.  Define $f \colon \R^{12} \to \R$ by \eqref{eq:f-def}. Then
$f$ is a radial Schwartz function satisfying $\ft f = f$ and
\begin{equation}
\label{eq:f-real}
f(x) = \sin\mathopen{}\big(\pi |x|^2/2\big)^2\mathclose{} \int_0^\infty \psi(it) e^{-\pi |x|^2 t}\, dt
\end{equation}
for $|x| > \sqrt{2}$.

It follows from \eqref{eq:psi-def} that
\begin{equation} \label{eq:positive}
\psi(it) > 0
\end{equation}
for all $t>0$, because $\Theta_{00}(it)$, $\Theta_{01}(it)$, and
$\Theta_{10}(it)$ are all real, while $0 < \Delta(it) < 1$.  Thus,
\eqref{eq:f-real} implies that $f(x) \ge 0$ for $|x| > \sqrt{2}$, with double
roots at $|x| = \sqrt{2j}$ for integers $j\ge 2$ and no other roots in this
range.

For comparison, the quasimodular form inequalities that play the same role as
\eqref{eq:positive} in \cite{V} and \cite{CKMRV} are obtained via
computer-assisted proofs. The reason for this discrepancy is that those
proofs combine $+1$ and $-1$ eigenfunctions, which introduces technical
difficulties. If all one wishes to prove is that $\bA_-(8) = \sqrt{2}$ and
$\bA_-(24)=2$, then one can avoid computer assistance.  Specifically, the
formula (3.1) in \cite{CKMRV} is visibly positive in the same sense as our
formula \eqref{eq:psi-def}, and while that is not true for formula (46) in
\cite{V}, it can be rewritten so as to be visibly positive (see, for example,
the corresponding formula in \cite{C2}).

To analyze the behavior of $f(x)$ with $0 \le |x| \le \sqrt{2}$, we can
simply cancel the growth of $\psi(it)$.  The series \eqref{eq:fourier} shows
that
\[
\psi(it) = e^{2\pi t} - 264 + O\big(e^{-\pi t}\big)
\]
as $t \to \infty$.  For $|x| > \sqrt{2}$, we obtain the new formula
\[
f(x) = \sin\mathopen{}\big(\pi |x|^2/2\big)^2\mathclose{} \left(\frac{528-263|x|^2}{\pi |x|^2 (|x|^2-2)} + \int_0^\infty \big(\psi(it)-e^{2\pi t} + 264\big) e^{-\pi |x|^2 t}\, dt\right)
\]
from \eqref{eq:f-real}, and the integral in this formula now converges for
all $x$. It follows from \eqref{eq:f-def} that $f(x)$ is a holomorphic
function of $|x|$; thus, the new formula must agree with the old one for all
$x$ by analytic continuation.

The term
\[
\sin\mathopen{}\big(\pi |x|^2/2\big)^2\mathclose{} \int_0^\infty \big(\psi(it)-e^{2\pi t} + 264\big) e^{-\pi |x|^2 t}\, dt
\]
vanishes to second order at $|x|=\sqrt{2j}$ for all $j \ge 1$, and to fourth
order at the origin.  Thus, $f(x)$ must agree with
\[
\sin\mathopen{}\big(\pi |x|^2/2\big)^2\mathclose{} \left(\frac{528-263|x|^2}{\pi |x|^2 (|x|^2-2)} \right)
\]
to second order at $|x|=\sqrt{2}$ and to fourth order at the origin, and so
$f(x)$ has a single root at $|x|=\sqrt{2}$ and a double root at the origin.
More specifically,
\[
f(x) = \frac{\pi}{\sqrt{2}}(|x|-\sqrt{2}) + O\big((|x|-\sqrt{2})^2\big)
\]
as $|x| \to \sqrt{2}$, and
\[
f(x) = -66\pi|x|^2 + O\big(|x|^4\big)
\]
as $x \to 0$.

In particular, $f(0)=0$.  It follows that $f \in \A_+(12)$, and therefore
$\bA_+(12) \le \sqrt{2}$, as desired.  We have now proved all of the
assertions from Theorem~\ref{thm_dimension12}.

As the quadratic term $-66\pi|x|^2$ suggests, our construction of $f$ is
scaled so that it values are rather large.  For example, its minimum value
appears to be $f(x) \approx -23.8088$, achieved when $|x| \approx 0.557391$.
In Figure~\ref{figure:dim12}, we have plotted a more moderate scaling of this
function.

To arrive at the definition \eqref{eq:psi-def} of $\psi$, we began with the
Ansatz that $\psi \Delta$ should be a holomorphic modular form of weight $8$
for $\Gamma(2)$. Equivalently, it should be a linear combination of
$\Theta_{00}^{16}$, $\Theta_{00}^{12} \Theta_{01}^4$, $\Theta_{00}^8
\Theta_{01}^8$, $\Theta_{00}^4 \Theta_{01}^{12}$, and $\Theta_{01}^{16}$.
Imposing the constraint $z^{4}\psi(-1/z) + \psi(z+1) = \psi(z)$ eliminates
three degrees of freedom, which leaves just one degree of freedom, up to
scaling. The remaining constraint is that the coefficient of $e^{-\pi i z}$
in the Fourier expansion of $\psi(z)$ must vanish, and then $\psi$ is
determined modulo scaling. Finally, we rewrote the formula for $\psi$ to make
it visibly positive.

\section{The $-1$ eigenfunction uncertainty principle}
\label{section:-1uncertainty}

This section is devoted to the proof of Theorem~\ref{thm_facts_+-1}. We deal
only with the $-1$ case, because all the assertions in this theorem were
already proved in \cite{GOS} for the $+1$ case. First, we reduce determining
$\bA_-(d)$ to solving Problem~\ref{-1problem}.

\begin{lemma}\label{reduction_lemma}
For each $f\in \A_-(d)\setminus\{0\}$, there exists a radial function $g\in
\A_-(d)\setminus\{0\}$ such that $\ft g =-g$, $g(0)=0$, and $r(g) \leq
\sqrt{r(f)r(\ft f\,)}$.
\end{lemma}

\begin{proof}
If $f$ is not radial, then we average its rotations about the origin to
obtain a radial function without increasing $r(f)$ or $r(\ft f\,)$. Thus, we
can assume that $f$ is radial.  Note that this process cannot lead to the
zero function: if it did, then $f$ and $\ft f$ would both have compact
support and hence vanish identically.

The quantity $r(f)r(\ft f\,)$ is unchanged if we replace $f$ with $x \mapsto
f(\lambda x)$ for some $\lambda>0$. Thus, we can assume that $r(f)=r(\ft
f\,)$. Letting $g=f-\ft f$ we deduce that $g\in \A_-(d)$, $\ft g=-g$, and
$r(g)\leq r(f)$.  Again, $g$ cannot vanish identically, because $f$ and $-\ft
f$ are eventually nonnegative and would thus both have to have compact
support.

It remains to force $g(0)=0$, since a priori we can have $g(0)>0$. For $t>0$,
consider the auxiliary function
\begin{equation} \label{eq:phi-t}
\varphi_t(x) = \frac{e^{-t\pi|x|^2} - e^{-2t\pi|x|^2}}{t^{-d/2} -
(2t)^{-d/2}}.
\end{equation}
Then $\varphi_t\geq 0$, $\varphi_t(0)=0$, $\ft \varphi_t(0)=1$, and $\ft
\varphi_t(x) <0$ if $|x|^2\geq td\log(2)/\pi$. Choosing $t>0$ so that
$\sqrt{td\log(2)/\pi} = r(g)$, we deduce that the function $h = g  +
g(0)(\varphi_t -\ft \varphi_t)$ belongs to $\A_-(d)$, $\ft h = -h$, $h(0)=0$,
and $r(h)\leq r(g)$.  Finally, if $g(0)>0$, then $h(x) > g(x)$ for all
sufficiently large $x$, and thus $h$ is not the zero function.
\end{proof}

\subsection{Lower and upper bounds}

To obtain a lower bound for $\bA_-(d)$, we follow \cite{BCK,GOS}. Let $g\in
\A_-(d)\setminus\{0\}$ be a radial function satisfying $\ft g=-g$ and
$g(0)=0$, and assume without loss of generality that $\|g\|_1 = 1$.

Let $g^+ = \max\{g,0\}$ and $g^- = \max\{-g,0\}$, so that $g^+,g^- \ge 0$,
these functions are never positive at the same point, and $g=g^+-g^-$. Since
$\ft g(0)=0$,
\[
\int_{\R^d} g^+ = \int_{\R^d} g^-.
\]
Furthermore,
\[
\int_{\R^d} g^- = \int_{B^d_{r(g)}} g^-,
\]
where $B^d_{r(g)}$ is a $d$-dimensional ball of radius $r(g)$ and centered at
the origin, because $\{x\in \R^d:g(x)< 0\}\subseteq B^d_{r(g)}$. It follows
that
\[
 \int_{B^d_{r(g)}} g^- = 1/2,
\]
because $\|g\|_{1}=1$. Thus,
\begin{align*}
1/2 &\leq \vol\mathopen{}\big(B^d_1\big)\mathclose{}r(g)^d \|g\|_{\infty}\\
& \leq \vol\mathopen{}\big(B^d_1\big)\mathclose{}r(g)^d \|\ft g\|_{1}\\
& = \vol\mathopen{}\big(B^d_1\big)\mathclose{}r(g)^d \|g\|_{1} \\
& = \vol\mathopen{}\big(B^d_1\big)\mathclose{}r(g)^d,
\end{align*}
and we conclude that
\begin{equation}
\label{lowerbound-1} \bA_-(d) \geq \left(\frac{1}{2 \vol\mathopen{}\big(B^d_1\big)\mathclose{}}\right)^{1/d} =
\frac{\Gamma(d/2+1)^{1/d}}{2^{1/d}\sqrt{\pi}} > \sqrt{\frac{d}{2\pi e}}.
\end{equation}

Next we prove an upper bound for $\bA_-(d)$. Let
\[
L_n^\nu(z) = \sum_{j=0}^n \binom {n+\nu}{n-j} \frac{(-z)^j}{j!}
\]
be the generalized Laguerre polynomial of degree $n$ with parameter $\nu>-1$.
When $\nu = d/2-1$, the functions $\psi_n^\nu \colon \R^d \to \R$ defined by
\begin{equation}
\label{def:phi-nu-n}
\psi_n^\nu(x) = L_n^\nu(2\pi |x|^2)e^{-\pi|x|^2}
\end{equation}
form a orthogonal basis for the space of radial functions in $L^2(\R^d)$, and
they are eigenfunctions for the Fourier transform:
\[
\ft \psi_n = (-1)^n\psi_n.
\]
(See, for example, Lemma~10 in \cite{GOS}.)

Let
\begin{align*}
p(z) &= L^\nu_1(z)L^\nu_3(0) - L^\nu_3(z)L^\nu_1(0)\\
& =
\frac{(1+\nu)}{6}z\left({2(3+\nu)(2+\nu)}  - 3(3+\nu)z + {z^2}\right).
\end{align*}
The roots of this polynomial are $0$ and
\[
\frac{3\nu+9\pm \sqrt{33+14\nu+\nu^2}}{2},
\]
and it is positive beyond the largest of these roots. If $\nu=d/2-1$, then
the largest root takes the form
\[
\frac{3d/2+6 + \sqrt{20+6d+d^2/4}}{2}.
\]
Now the function $g \colon \R^d \to \R$ defined by
\begin{align*}
g(x) & = \psi_1^\nu(x)\psi_3^\nu(0) - \psi_3^\nu(x)\psi_1^\nu(0)\\
& = p(2\pi|x|^2) e^{-\pi|x|^2}
\end{align*}
is radial, belongs to $\A_-(d)$, and satisfies $\ft g = - g$ and $g(0)=0$.
Hence
\begin{equation}
\label{upperbound-1} \bA_-(d) \leq \sqrt{\frac{3d/2+6 +
\sqrt{20+6d+d^2/4}}{4\pi}}  = \big(1+O(d^{-1/2})\big)\sqrt{\frac{d}{2\pi}}.
\end{equation}

Estimates \eqref{lowerbound-1} and \eqref{upperbound-1} imply that
$\bA_-(d)/\sqrt{d}$ is bounded above and below by positive constants, as
desired.  In particular, the lower bound is $1/\sqrt{2\pi e}$, and the upper
bound is at most $1$ except for $d=1$, in which case we can use $\bA_-(1)=1$
to obtain an upper bound of $1$.

We believe that the upper bound \eqref{upperbound-1} cannot be improved if we
replace $p$ with any polynomial of bounded degree, in the following sense.
For $N \ge 3$ and $s = \pm 1$, let $\bA_{s,N}(d)$ be the infimum of $r(g)$
over all nonzero $g \colon \R^d \to \R$ such that $\ft g = sg$, $g(0)=0$, and
$g$ is of the form
\[
g(x) = p(2\pi |x|^2) e^{-\pi |x|^2},
\]
where $p$ is a polynomial of degree at most $N$.  (The restriction to $N \ge
3$ ensures that such a function exists.)

\begin{conjecture} \label{conj:fixeddegree}
For fixed $N \ge 3$ and $s = \pm 1$,
\[
\lim_{d \to \infty} \frac{\bA_{s,N}(d)}{\sqrt{d}} = \frac{1}{\sqrt{2\pi}}.
\]
\end{conjecture}

However, the upper bound for $\bA_-(d)$ can be improved using other
functions. In particular, we can make use of the auxiliary functions $f$
constructed in \cite{CZ} for the linear programming bound in high dimensions.
If we set $g = \ft f - f$, then one can show that
\[
r(g) \le (0.3194\ldots + o(1)) \sqrt{d}
\]
as $d \to \infty$.  The number $0.3194\ldots$ is derived from the
Kabatiansky-Levenshtein bound for sphere packing, and the construction in
\cite{CZ} shows how to obtain that bound via the linear programming bound.
The precise number is rather complicated, but it can be characterized as
follows. Let $\theta = 1.0995\ldots$ be the unique root of
\[
2\log(\sec(\theta) + \tan(\theta)) = \sin(\theta) + \tan(\theta)
\]
in the interval $(0,\pi/2)$, and let
\[
c = \frac{\sin(\theta/2) \cot(\theta) e^{\sec(\theta)/2}}{\sqrt{2\pi}} = 0.3194\ldots.
\]
Then
\[
r(g) \le (c+o(1)) \sqrt{d}
\]
as $d \to \infty$, and hence
\[
\limsup_{d \to \infty} \frac{\bA_-(d)}{\sqrt{d}}\le c.
\]
We do not know how to prove the corresponding bound for $\bA_+(d)$, although
we believe it should be true, as it would follow from
Conjecture~\ref{conj:limit}.

\subsection{Existence of extremizers.} The existence proof for extremizers
with $s=-1$ is almost identical to the proof of the $+1$ case in
\cite[Section~6]{GOS}. We briefly outline the proof here for completeness.
Let $f_n \in \A_-(d)\setminus\{0\}$ be an extremizing sequence; that is,
$\sqrt{r(f_n)r(\ft f_n)} \searrow \bA_-(d)$ as $n\to \infty$. By
Lemma~\ref{reduction_lemma} we can assume that $\ft f_n = -f_n$ and
$f_n(0)=0$, and hence $r(f_n)\searrow \bA_-(d)$. We can also assume that
$\|f_n\|_{1} = 1$ for all $n$. In particular, since $\ft f_n=-f_n$, we have
\[
\|f_n\|_{2}^2 = \int_{\R^d} |f_n|^2 \le \|f_n\|_\infty \cdot \|f_n\|_1 \leq \|\ft f_n\|_1 \cdot \|f_n\|_{1} = 1.
\]
Because the unit ball in $L^2(\R^d)$ is weakly compact, we can assume that
$f_n$ converges weakly to some function $f\in L^2(\R^d)$. Because ${\A}_-(d)$
is convex, we can apply Mazur's lemma to assume furthermore that $f_n$
converges almost everywhere and in $L^2(\R^d)$ to $f$. Thus, necessarily we
have $\ft f=-f$ and $r(f)\leq \bA_-(d)$. Since $\|f_n\|_{\infty}\leq \|\ft
f_n\|_{1} = \|f_n\|_{1}= 1$ and $r(f_n)$ is decreasing, we can apply Fatou's
lemma for $g_n=\mathbf{1}_{B^d_{r(f_1)}}  + f_n\geq 0$ to deduce that $f\in
L^1(\R^d)$ and $\ft f(0)\leq 0$. Hence, $f(0)\geq 0$.  We now use Jaming's
high-dimensional version \cite{J} of Nazarov's uncertainty principle \cite{N}
to deduce, exactly as in \cite[Lemma~23]{GOS}, that there exists $K<0$ such
that for all $n$,
\[
\int_{B^d_{r(f_n)}} f_n \leq  K.
\]
(Alternatively, we can use Proposition~2.6 from \cite{BD}, which tells us
less about the constant $K$ but has a simpler proof.) Fatou's lemma implies
that $f$ satisfies the same estimate, and hence is not identically zero. We
conclude that $f \in \A_-(d)$, $\ft f = -f$, and $r(f)\leq \bA_-(d)$, and
thus $r(f) = \bA_-(d)$. Finally, we must have $f(0)=0$, since otherwise the
proof of Lemma~\ref{reduction_lemma} would produce a better function.

\subsection{Infinitely many roots.}\label{root_property}
All that remains to prove is that the extremizers have infinitely many roots.
The proof follows the ideas of \cite[Section~6.2]{GOS} for the $+1$ case. If
$f\in \A_-(d)$ satisfies $\ft f=-f$ and $f(0)=0$ and vanishes at only
finitely many radii beyond $r(f)$, then we find a perturbation function $g\in
\A_-(d)$ satisfying $\ft g = -g$ and $g(0)=0$ such that $r(f+\varepsilon g) <
r(f)$ for small $\varepsilon>0$; thus, $f$ cannot be extremal. In \cite{GOS},
the construction of $g$ varies between the cases $d=1$ (using the Poincar\'e
recurrence theorem) and $d\geq 2$ (using a trick involving Laguerre
polynomials). However, thanks to the Poisson summation formula, every
extremal function $f\in \A_-(1)$ with $\ft f=-f$ and $f(0)=0$ must vanish at
the integers. Thus, we only need to prove our assertion for $d\geq 2$.

In fact, we will rule out the possibility that an extremizer $f$ is
eventually positive. Then applying this proof to the radialization of $f$
will show that $f$ must vanish on spheres of arbitrarily large radius. Thus,
let $f \in \A_-(d)$ be such that $\ft f = -f$, $f(0)=0$, and $f(x)>0$ for
$|x| \ge R$.  We must show that $r(f) > \bA_-(d)$.

Let $\varphi_t$ be the function defined in \eqref{eq:phi-t} with $t \in
(0,1)$ chosen so that
\[
\sqrt{td\log(2)/\pi}<r(f),
\]
and let $\psi = \varphi_t - \ft \varphi_t$. Then $\ft \psi =- \psi$,
$\psi(0)=-1$, and $\psi(x) > 0$ for $|x|\ge r(f)$. This function almost works
as a possible perturbation $g$, but it needs to be fixed at the origin
without changing its eventual nonnegativity. To do so, let $\nu=d/2-1$ and
consider the function
\[
g_n = \psi + \frac{\psi^\nu_{2n+1}}{\psi^\nu_{2n+1}(0)},
\]
where $\psi^\nu_{2n+1}$ is the eigenfunction defined in \eqref{def:phi-nu-n}.
Now $\ft g_n = -g_n$, $g_n(0)=0$, and $g_n$ is eventually positive for each
$n\geq 0$, because $t<1$ implies that $\psi^\nu_{2n+1}$ decays faster than
$\psi$.

As observed in \cite{GOS}, for $d \ge 2$ the eigenfunctions
$\psi^\nu_j/\psi^\nu_j(0)$ converge to zero uniformly on all compact subsets
of $\R^d\setminus\{0\}$ as $j \to \infty$; the proof amounts to Fej\'er's
asymptotic formula for Laguerre polynomials \cite[Theorem~8.22.1]{S}.  Using
this convergence, let $n$ be large enough that $g_n(x) > 0$ for $|x| \in
[r(f),R]$, and then choose $R'$ so that $g_n(x) > 0$ for $|x| \ge R'$. Let
$m=\min\{|f(x)| : R\leq |x| \leq R'\}$, $M=\max\{|g_{n}(x)|: x\in \R^d\}$,
and $0<\varepsilon<m/M$. Then the perturbation $f_\varepsilon = f +
\varepsilon g_n$ satisfies $f_\varepsilon(x) > 0$ for $|x| \ge r(f)$.  Thus,
$r(f_\varepsilon) < r(f)$, which means $f$ cannot be extremal. This completes
the proof of Theorem~\ref{thm_facts_+-1}.

\section{Numerical evidence}
\label{section:numerics}

To explore how $\bA_+(d)$ behaves, we numerically optimized functions $g
\colon \R^d \to \R$ satisfying the conditions of Problem~\ref{+1problem}.
Readers who wish to examine this data can obtain our numerical results from
\cite{CG}.

In our calculations we always choose $g$ to be of the form $g(x) =
p(2\pi|x|^2) e^{-\pi |x|^2}$, where $p$ is a polynomial in one variable of
degree at most $4k+2$, which means $p$ has $4k+2$ degrees of freedom modulo
scaling. The constraint $g(0)=0$ eliminates one degree of freedom, and one
can check using the Laguerre eigenbasis that the constraint $\ft g = g$
eliminates $2k+1$ degrees of freedom. To control the remaining $2k$ degrees
of freedom, we specify $k$ double roots at radii $\rho_1 < \dots < \rho_k$.
We then attempt to choose the radii $\rho_1,\dots,\rho_k$ so as to minimize
$r(g)$.  To do so, we iteratively optimize the choice of radii for successive
values of $k$, by making an initial guess based on the previous value of $k$
and then improving the guess using multivariate Newton's method. Each choice
of $\rho_1,\dots,\rho_k$ proves an upper bound for $\bA_+(d)$, and we hope to
approximate $\bA_+(d)$ closely as $k$ grows. (Note that if
Conjecture~\ref{conj:fixeddegree} holds, then we cannot obtain improved
bounds if $k$ remains bounded for large $d$.) This method was first applied
by Cohn and Elkies \cite[Section~7]{CE} to $\bA_-(d)$, with a simpler
optimization algorithm. Cohn and Kumar \cite{CK} replaced that algorithm with
Newton's method, and we made use of their implementation.

We have no guarantee that the numerical optimization will converge to even a
local optimum for any given $d$ and $k$, or that the resulting bounds will
converge to $\bA_+(d)$ as $k \to \infty$.  Indeed, we quickly ran into
problems when $d \le 2$, and eventually for $d=3$ and $4$ as well, but for $5
\le d \le 128$ we arrived at the global optimum for each $k \le 64$. These
calculations are what initially led us to believe that $\bA_+(12)=\sqrt{2}$.

Our numerical calculations are generally not rigorous: although we believe we
have used more than sufficient precision, we cannot bound the error from the
use of floating-point arithmetic. However, we have used exact rational
arithmetic to prove all the numerical upper bounds for $\bA_s(d)$ we report
in this paper.\footnote{The non-sharp cases from Table~\ref{table:data} are
straightforward to check rigorously, while the inequality $\bA_+(28) <
1.98540693489105$ requires more work because it uses a higher-degree
polynomial with more complicated coefficients. We have proved it using the
techniques and code from Appendix~A of \cite{CK}.} Thus, they are genuine
theorems, while our numerical assertions about summation formulas have not
been rigorously proved.

\begin{table}
\caption{Upper bounds for $\bA_+(d)$ and $\bA_-(d-4)$.} \label{table:data}
\begin{tabular}{rccrcc}
\toprule
$d$ & $\bA_+(d)$ & $\bA_-(d-4)$ & $d$ & $\bA_+(d)$ & $\bA_-(d-4)$\\
\cmidrule(r){1-3} \cmidrule(l){4-6}
$1$ & $0.572990$ & & $17$ & $1.619692$ & $1.627509$\\
$2$ & $0.756207$ & & $18$ & $1.657044$ & $1.665874$\\
$3$ & $0.887864$ & & $19$ & $1.693390$ & $1.703115$\\
$4$ & $0.965953$ & & $20$ & $1.728806$ & $1.739328$\\
$5$ & $1.036454$ & $1$ & $21$ & $1.763360$ & $1.774593$\\
$6$ & $1.101116$ & $1.074570$ & $22$ & $1.797112$ & $1.808982$\\
$7$ & $1.161109$ & $1.141962$ & $23$ & $1.830115$ & $1.842559$\\
$8$ & $1.217275$ & $1.203808$ & $24$ & $1.862417$ & $1.875378$\\
$9$ & $1.270241$ & $1.261244$ & $25$ & $1.894060$ & $1.907490$\\
$10$ & $1.320483$ & $1.315083$ & $26$ & $1.925084$ & $1.938938$\\
$11$ & $1.368375$ & $1.365923$ & $27$ & $1.955522$ & $1.969763$\\
$12$ & $\sqrt{2}$ & $\sqrt{2}$ & $28$ & $1.985407$ & $2$\\
$13$ & $1.458239$ & $1.460307$ & $29$ & $2.014769$ & $2.029684$\\
$14$ & $1.500647$ & $1.504478$ & $30$ & $2.043633$ & $2.058842$\\
$15$ & $1.541603$ & $1.546952$ & $31$ & $2.072024$ & $2.087503$\\
$16$ & $1.581246$ & $1.587911$ & $32$ & $2.099965$ & $2.115691$\\
\bottomrule
\end{tabular}
\end{table}

Table~\ref{table:data} shows our upper bounds for $\bA_+(d)$ for $1 \le d \le
32$, together with $\bA_-(d-4)$ for comparison (taken from \cite{C2}).  The
shift by $4$ approximately aligns the columns, with the best case being
$\bA_+(12) = \bA_-(8) = \sqrt{2}$. We have no conceptual explanation for this
alignment, but it fits conveniently with the sign in
Proposition~\ref{prop:psi}, and it supports our conjecture that
\[
\lim_{d \to \infty} \frac{\bA_+(d)}{\sqrt{d}} = \lim_{d \to \infty} \frac{\bA_-(d)}{\sqrt{d}}.
\]
The convergence to this limit is slow enough that it is difficult to estimate
the limit accurately from numerical data.

For $d \le 2$ our numerical methods perform poorly, for the reasons described
below.  For $d=3$ the bound for $\bA_+(d)$ in Table~\ref{table:data} is
obtained using $k=27$, and for $d \ge 4$ we use $k=32$.  In particular, we
deliberately use a smaller value of $k$ than the limits of our computations
for $d \ge 4$, so that we can use data from larger $k$ to estimate the rate
of convergence.  These computations suggest the following conjecture.

\begin{conjecture} \label{conj:table-optimal}
For $3 \le d \le 32$, the upper bounds for $\bA_+(d)$ and $\bA_{-}(d-4)$ in
Table~\ref{table:data} are sharp, except for an error of at most $1$ in the
last decimal digit shown.
\end{conjecture}

In each case with $d \ge 3$, we can use a summation formula to check that we
have found the optimal bound for the given values of $d$ and $k$; we explain
how this is done in Section~\ref{section:summation-formulas}.  However, we do
not know how quickly the bounds converge as $k \to \infty$, or whether they
indeed converge to $\bA_s(d)$ at all.  Our confidence in
Conjecture~\ref{conj:table-optimal} comes from comparing the bounds for $32
\le k \le 64$ when $d \ge 5$.  They seem to have converged to this number of
digits, but of course we cannot rule out convergence to the wrong limit.

The approximation $\bA_+(d) \approx \bA_-(d-4)$ and equality $\bA_+(12) =
\bA_-(8) = \sqrt{2}$ raise the question of whether the other exact values
$\bA_-(1)=1$, $\bA_-(2) = (4/3)^{1/4}$ (conjecturally), and $\bA_-(24)=2$ are
also mirrored by $\bA_+$.  That turns out not to be the case:
Table~\ref{table:data} strongly suggests that $\bA_+(5) > 1$ and $\bA_+(6) >
(4/3)^{1/4}$, and it proves that $\bA_+(28) < 2$.  The case of $\bA_+(28)$ is
particularly disappointing, because it might have stood in the same
relationship to $\bA_+(12)$ as the Leech lattice does to the $E_8$ root
lattice. We have found no case other than $d=12$ for which we can guess the
exact value of $\bA_+(d)$.

Taking $k=128$ shows that $\bA_+(28) < 1.98540693489105$, and again we
believe that all these digits agree with $\bA_+(28)$ except the last.  This
upper bound for $\bA_+(28)$ seems discouragingly complicated, but the
underlying root locations display remarkable behavior, shown in
Table~\ref{table:data28}. The table leads us to the following conjecture:

\begin{table}
\caption{Approximations to $r(g)^2, \rho_1^2, \rho_2^2, \dots, \rho_{31}^2$
when $d=28$ and $k=128$. We view these numbers as approximations to the
squared radii for the roots of a function achieving $\bA_+(28)$.} \label{table:data28}
\begin{tabular}{llll}
\toprule
$3.9418406971135$ & $20.000001150214$ & $35.999999987965$ & $52.000000000234$\\
$5.9937066227310$ & $21.999999768273$ & $37.999999967198$ & $54.000000000902$\\
$8.0001376275780$ & $23.999999651853$ & $40.000000012100$ & $55.999999999543$\\
$10.000148227366$ & $25.999999804782$ & $42.000000017800$ & $58.000000002140$\\
$12.000008052312$ & $28.000000118205$ & $43.999999995225$ & $60.000000000589$\\
$13.999980992905$ & $30.000000112036$ & $46.000000002272$ & $61.999999999086$\\
$15.999998782377$ & $31.999999979813$ & $48.000000000644$ & $63.999999999805$\\
$18.000002092309$ & $33.999999997483$ & $49.999999993657$ & $65.999999999746$\\
\bottomrule
\end{tabular}
\end{table}

\begin{conjecture}
\label{conjecture:28} There exists a radial Schwartz function $g \in \A_+(28)
\setminus \{0\}$ with $\ft g = g$, $g(0)=0$, and $r(g) = \bA_+(28)$, and
whose nonzero roots are at radii $\sqrt{2j + o(1)}$ as $j \to \infty$,
starting with $j=2$.
\end{conjecture}

This pattern is reminiscent of \cite[Section~7]{CM}, as well as the behavior
of $\bA_\pm(d)$ in other cases, but it is a particularly striking example. We
expect that Conjecture~\ref{conjecture:28} is true, but a weaker conjecture
consistent with the data is that there exists some $\varepsilon<1$ such that
the squared radii are within $\varepsilon$ of successive even integers.

For comparison, \cite{CKMRV} constructs a function achieving $\bA_-(24)$
whose nonzero roots are exactly at $\sqrt{2j}$ with $j \ge 2$.  Our best
guess is that the function achieving $\bA_+(28)$ is given by a primary term
that has these exact roots, plus one or more secondary terms that perturb the
roots but do not substantially change them.  If that is the case, then
perhaps one can describe this function explicitly and thereby characterize
$\bA_+(28)$ exactly.  However, we have not been able to guess or derive such
a formula.

Another mystery is the behavior of $\bA_+(d)$ for $d \le 2$.  In these
dimensions we quickly run into cases in which the last sign change $r(g)$ is
not a continuous function of $\rho_1,\dots,\rho_k$ at the optimum, and this
lack of continuity ruins our numerical algorithms. (Instead, we resort to
linear programming, which is much slower.) Of course it is no surprise that
the last sign change is discontinuous at some points, because a small
perturbation of a polynomial can convert a double root to two single roots,
or even create a new root if the degree increases. However, we do not expect
this behavior to occur generically.  In particular, it cannot occur if
$\deg(p)=4k+2$ and $g$ has no double roots beyond the $k$ double roots we
have forced to occur.

When $d=2$, even the case $k=1$ is problematic.  Specifically, one can check
that the optimal value $r(g) = \sqrt{2/\pi}$ is achieved by setting $\rho_1 =
\sqrt{3/\pi}$. As $\rho_1$ approaches $\sqrt{3/\pi}$ from the left, $r(g)$
decreases towards $\sqrt{2/\pi}$, but it increases towards infinity as
$\rho_1$ approaches $\sqrt{3/\pi}$ from the right.  This discontinuity occurs
because the leading coefficient of the polynomial $p$ vanishes when $\rho_1 =
\sqrt{3/\pi}$.  The leading coefficient also vanishes at the best choices of
$\rho_1,\dots,\rho_k$ we have found for $2 \le k \le 4$, while the case $k=5$
suffers from a different problem: the resulting polynomial has six double
roots, rather than just five, and again the location of the last sign change
is discontinuous.

When $d=1$, there are no problems for $k \le 2$, and the leading coefficient
vanishes for $k=3$.  For $k=4$, we find an extra double root, but there is no
discontinuity when $k=5$.

In Table~\ref{table:data} we have reported the bound using $k=5$ for $d\le
2$.  We believe that we have approximated the true optima for $k=5$, but the
bounds almost certainly do not agree with $\bA_+(d)$ to the full six digits
shown, unlike Conjecture~\ref{conj:table-optimal}.

We have not observed a discontinuity near the optimum in any other dimension.
However, when $d=3$ we cannot find a local optimum with $k=28$, because the
largest root tends to infinity in our calculations.  Computations carried out
by David de~Laat indicate that the optimum occurs at a singularity and the
resulting discontinuity is interfering with our algorithms.  When $d=4$ we
run into a similar problem at $k=36$. We do not know whether this phenomenon
is limited to $d \le 4$.

\section{Summation formulas}
\label{section:summation-formulas}

We do not know how to obtain the hypothetical summation formulas described in
Conjecture~\ref{conjecture:summation}.  Aside from $\bA_-(2)$ and the four
cases that have been solved exactly (namely $\bA_-(1)$, $\bA_-(8)$,
$\bA_+(12)$, and $\bA_-(24)$), we have not found any summation formulas that
come close to matching our upper bounds. However, in many cases we can
compute optimal summation formulas for polynomials of a fixed degree. For $d
\ge 3$, these formulas show that we have found the optimal polynomials for
each fixed $k$ in our computations in Section~\ref{section:numerics}, and we
believe that when $k$ is large they should approximate the ultimate summation
formulas. For example, Table~\ref{table:summation-formula} is based on
calculations with $k=128$.

Recall that our numerical method uses the Laguerre eigenbasis.  If we are
bounding $\bA_s(d)$, we let $\nu = d/2-1$ and
\[
q_j = \begin{cases} L_{2j}^\nu & \textup{if $s=1$, and}\\
L_{2j+1}^\nu & \textup{if $s=-1$.}
\end{cases}
\]
Then our method seeks a linear combination $p$ of $q_0,q_1,\dots,q_{2k+1}$
that vanishes at $0$ and minimizes $r(p)$; using the function $f(x) = p(2\pi
|x|^2)e^{-\pi |x|^2}$, we conclude that $\bA_s(d) \le \sqrt{r(p)/(2\pi)}$,
where
\[
r(p) = \inf {} \{R \ge 0: \text{$p(x)$ has the same sign for $x\geq R$}\}.
\]
(Unlike earlier, we require only $x \ge R$ in the definition of $r(p)$,
rather than $|x| \ge R$, because we care only about the right half-line.) To
construct $p$, we impose double roots at locations $\rho_1,\dots,\rho_k$, and
then choose these locations so as to minimize $\rho_0 := r(p)$.  Note that in
our notation here, $\rho_i$ denotes what would have been called $2\pi
\rho_i^2$ in Section~\ref{section:numerics}.

To obtain a summation formula, we will need to impose some non-degeneracy
conditions. We will assume that $0 < \rho_0 < \rho_1 < \dots < \rho_k$, and
that $p$ is uniquely determined among linear combinations of
$q_0,\dots,q_{2k+1}$ by the following conditions:
\begin{enumerate}
\item $p(0)=0$,

\item $p(\rho_i) = p'(\rho_i)=0$ for $1 \le i \le k$, and

\item the coefficient of $q_{2k+1}$ is $1$.
\end{enumerate}
We assume furthermore that $p$ has roots of order exactly~$1$ at~$\rho_0$ and
exactly~$2$ at $\rho_1,\dots,\rho_k$, and no other real roots greater than
$\rho_0$. Finally, we assume that we have found a strict local minimum for
$r(p)$; in other words, $r(p)$ increases if we perturb $\rho_1,\dots,\rho_k$.

These assumptions cannot always be satisfied.  For example, when
$(s,d,k)=(1,2,1)$ the coefficient of $q_{2k+1}$ vanishes. However, for $d>2$
they are satisfied in every case in which we have found a local minimum. See
Table~\ref{table:computed} for a list.

\begin{table}
\caption{Values of $k$ for which we have numerically computed a local minimum
and the corresponding summation
formula to one hundred decimal places. When $(s,d) = (1,1)$, $(1,2)$, $(1,3)$,
$(1,4)$, or $(-1,3)$,
we believe the next value of $k$ does not work
(i.e., there is no local optimum satisfying our non-degeneracy conditions);
otherwise, the table simply leaves off where we
stopped computing.} \label{table:computed}
\begin{tabular}{ccccccc}
\toprule
$s$ & $d$ & $k$ & \qquad & $s$ & $d$ & $k$\\
\cmidrule(r){1-3} \cmidrule(l){5-7}
$1$ & $1$ & $1$, $2$, $5$ & & $-1$ & $1$ & $1$--$64$\\
$1$ & $2$ & --- & & $-1$ & $2$ & $1$--$64$\\
$1$ & $3$ & $1$--$27$ & & $-1$ & $3$ & $1$--$20$, $26$--$31$\\
$1$ & $4$ & $1$--$35$ & & $-1$ & $4$--$128$ & $1$--$64$\\
$1$ & $5$--$128$ & $1$--$64$\\
$1$ & $28$ & $1$--$128$\\
\bottomrule
\end{tabular}
\end{table}

\begin{proposition} \label{proposition:summation-formula}
Under the hypotheses listed above, up to scaling there are unique
coefficients $c_0,\dots,c_{k+1}$, not all zero, such that
\[
\sum_{i=0}^k c_i g(\rho_i) + c_{k+1} g(0) = 0
\]
for every linear combination $g$ of $q_0,\dots,q_{2k+1}$.  Furthermore,
$c_0,\dots,c_k$ are nonzero and have the same sign.  If $s=1$, then $c_{k+1}$
is nonzero and has the opposite sign.
\end{proposition}

We prove this proposition below.  It is a polynomial analogue of the
summation formula \eqref{eq:general-summation2} (with the Gaussian factors
from the Laguerre eigenbasis implicitly incorporated into the coefficients
$c_i$), and it is reminiscent of Gauss-Jacobi quadrature in that it holds on
a $(2k+2)$-dimensional space despite using only $k+2$ coefficients.

\begin{corollary} \label{cor:optimal}
Any linear combination $g$ of $q_0,\dots,q_{2k+1}$ with $g(0)=0$ and $r(g) <
\rho_0$ must vanish identically, and $p$ is the unique linear combination
achieving $r(p) = \rho_0$, up to scaling.
\end{corollary}

In other words, although we have assumed only a strict local minimum for the
last sign change among polynomials with $k$ double roots, we have found the
global minimum among polynomials with no such restriction.  For example, when
$s=1$ and $k=64$, we find that $p$ is the best possible polynomial of degree
at most $4k+2=258$. This phenomenon not only certifies our numerics by
establishing matching lower bounds, but also helps explain why our algorithms
perform well: degeneracy is the only way to get stuck in a local optimum.

\begin{proof}[Proof of Corollary~\ref{cor:optimal}]
Suppose $g$ is a linear combination of $q_0,\dots,q_{2k+1}$ with $r(g) \le
\rho_0$, $g(0) = 0$, and $g(z) \ge 0$ for large $z$. By
Proposition~\ref{proposition:summation-formula},
\[
\sum_{i=0}^k c_i g(\rho_i) = - c_{k+1} g(0) = 0.
\]
Because $\rho_0 \ge r(g)$, all of $g(\rho_0),\dots,g(\rho_k)$ must be
nonnegative. It follows that $g$ must vanish at $\rho_0,\dots,\rho_k$, since
$c_0,\dots,c_k$ are nonzero and have the same sign. Furthermore,
$\rho_1,\dots,\rho_k$ must be roots of even order, since otherwise $g$ would
change sign beyond $r(g)$. However, we have assumed that the equations
$g(0)=0$, $g(\rho_i)=0$, and $g'(\rho_i)=0$ for $1 \le i \le k$ determine $g$
up to scaling.  Thus $g$ must be proportional to $p$, and the only way to
achieve $r(g) < r(p)$ is if $g$ vanishes identically.
\end{proof}

It will prove convenient to distinguish between $\rho_1,\dots,\rho_k$ and
perturbations of these points. For that purpose, we fix $\rho_1,\dots,\rho_k$
as the values described above, while $\rhot_1,\dots,\rhot_k$ are variables
taking values in some neighborhood of $\rho_1,\dots,\rho_k$.

The proof of Proposition~\ref{proposition:summation-formula} involves
carefully studying how different quantities behave as functions of
$\rhot_1,\dots,\rhot_k$.  We can set up simultaneous linear equations to
determine the coefficients of $q_0,\dots,q_{2k+1}$ as follows. Write $\alpha
= (\alpha_j)_{0 \le j \le 2k+1}$ for the column vector of coefficients (all
vectors will be column vectors unless otherwise specified, sometimes indexed
starting with $0$ and sometimes with $1$), and define the entries of the
matrix $M = (M_{i,j})_{0 \le i,j \le 2k+1}$ as follows:
\[
M_{i,j} = \begin{cases}
q_j(0) & \textup{for $i=0$,}\\
q_j(\rhot_i) & \textup{for $1 \le i \le k$,}\\
q_j'(\rhot_{i-k}) & \textup{for $k+1 \le i \le 2k$, and}\\
\delta_{j,2k+1} & \textup{for $i=2k+1$.}
\end{cases}
\]
Let $v = (\delta_{i,2k+1})_{0 \le i \le 2k+1}$.  Then the equation $M \alpha
= v$ expresses the constraints that $\sum_{j=0}^{2k+1} \alpha_j q_j$ vanishes
at $0$, vanishes to second order at $\rhot_1,\dots,\rhot_k$, and has
$\alpha_{2k+1}=1$.

We write $\rhot = (\rhot_1,\dots,\rhot_k)$ and $\rho =
(\rho_1,\dots,\rho_k)$.  When necessary to avoid confusion, we write
$M(\rhot)$ for the matrix depending on $\rhot$, $\alpha(\rhot)$ for the
solution of $M(\rhot) \alpha = v$ if $M(\rhot)$ is invertible, and
$p_{\rhot}$ for the corresponding linear combination $\sum_{j=0}^{2k+1}
\alpha_j q_j$ of $q_0,\dots,q_{2k+1}$.  Thus, the polynomial $p$ discussed
above amounts to $p_{\rho}$.

We have assumed that $M(\rho)$ is invertible, which means that
$\alpha(\rhot)$ and $p_{\rhot}$ are smooth functions of $\rhot$ defined on
some neighborhood of $\rho$.  Because $p_{\rho}$ has a single root at
$\rho_0$, $p_{\rhot}$ has a single root at some smooth function $\rhot_0$ of
$\rhot_1,\dots,\rhot_k$ with $\rhot_0(\rho) = \rho_0$, by the implicit
function theorem.  We will always assume that $\rhot$ is in a small enough
neighborhood of $\rho$ for this to be true.  Furthermore, our assumptions so
far imply that $r(p_{\rhot}) = \rhot_0$ for $\rhot$ in some neighborhood of
$\rho$, and again we restrict our attention to such a neighborhood.

Because of our assumption of local minimality, the function $\rhot_0$ must
have a stationary point at $\rho$. In other words,
\[
\frac{\partial \rhot_0}{\partial \rhot_i}(\rho) = 0
\]
for $1 \le i \le k$.  In addition, $\rhot_0 > \rho_0$ for $\rhot \ne \rho$ in
some small neighborhood of $\rho$ by strict local minimality.  Once again we
confine $\rhot$ to such a neighborhood.

\begin{lemma} \label{lemma:linearly-independent}
The vectors $\alpha(\rho)$ and $(\partial \alpha / \partial \rhot_i)(\rho)$
with $1 \le i \le k$ are linearly independent.
\end{lemma}

\begin{proof}
The vector $\alpha$ has $\alpha_{2k+1}=1$, while all the partial derivatives
$\partial \alpha / \partial \rhot_i$ vanish in that coordinate.  Thus, it
will suffice to show that the partial derivatives are linearly independent at
$\rho$, and because $M$ is invertible, we can examine $M (\partial \alpha /
\partial \rhot_i)$ instead of $\partial \alpha / \partial \rhot_i$.

Differentiating $M \alpha = v$ shows that
\[
M \frac{\partial \alpha}{\partial \rhot_i} = - \frac{\partial M}{\partial \rhot_i} \alpha.
\]
The matrix $\partial M / \partial \rhot_i$ vanishes except in rows $i$ and
$k+i$, and the entries of $(\partial M / \partial \rhot_i) \alpha$ in those
rows are $p_{\rhot}'(\rhot_i)$ and $p_{\rhot}''(\rhot_i)$, respectively. We
have $p_{\rhot}'(\rhot_i)=0$ by construction, but $p_{\rho}''(\rho_i) \ne 0$.
Thus, the vectors $(\partial M /
\partial \rhot_i)(\rho)\, \alpha(\rho)$ 
are linearly independent, as desired.
\end{proof}

\begin{lemma} \label{lemma:formula-exists}
There are real numbers $c_0,\dots,c_{k+1}$, not all zero, such that
\[
\sum_{i=0}^k c_i g(\rho_i) + c_{k+1} g(0) = 0
\]
for every linear combination $g$ of $q_0,\dots,q_{2k+1}$.
\end{lemma}

This lemma differs from Proposition~\ref{proposition:summation-formula} in
not asserting uniqueness or sign conditions for $c_0,\dots,c_{k+1}$.

\begin{proof}
Define the matrix
\[
T = (T_{i,j})_{\substack{0 \le i \le k+1\\  0 \le j \le 2k+1}}
\]
by
\[
T_{i,j} = \begin{cases} q_j(\rho_i) & \textup{for $0 \le i \le k$, and}\\
q_j(0) & \textup{for $i = k+1$}.
\end{cases}
\]
Then
\[
(c_0,\dots,c_{k+1})^{\top}T = \left(\sum_{i=0}^k c_i q_j(\rho_i) + c_{k+1} q_j(0)\right)^{\top}_{0 \le j \le 2k+1}
\]
for all row vectors $(c_0,\dots,c_{k+1})^{\top}$. Thus, the desired summation
formula amounts to a nonzero row vector in the kernel of right multiplication
by $T$.  To prove that such a vector exists, we will show that $\rank(T) <
k+2$.

It will suffice to find $k+1$ linearly independent vectors in the kernel of
left multiplication by $T$, because $(2k+2)-(k+1) < k+2$.  Those vectors will
be $\alpha(\rho)$ and $(\partial \alpha/\partial \rhot_i)(\rho)$ for $1 \le i
\le k$, which are linearly independent by
Lemma~\ref{lemma:linearly-independent}. All that remains is to prove that
they are in the kernel of $T$.

We have $T\alpha = (p_{\rhot}(\rho_0),\dots,p_{\rhot}(\rho_k),p_{\rhot}(0))$,
and thus $T \alpha(\rho) = 0$.  For the partial derivatives, we must show
that
\begin{equation} \label{eq:partial}
\sum_{j=0}^{2k+1} \frac{\partial \alpha_j}{\partial \rhot_i}(\rho) \, q_j(\rho_n)  = 0 
\end{equation}
for $0 \le n \le k$ and
\[
\sum_{j=0}^{2k+1} \frac{\partial \alpha_j}{\partial \rhot_i}(\rho) \, q_j(0)  = 0. 
\]
The latter equation follows from differentiating the identity
\[
\sum_{j=0}^{2k+1} \alpha_j q_j(0)  = 0.
\]
To prove \eqref{eq:partial}, we start with the fact that
\[
\sum_{j=0}^{2k+1} \alpha_j q_j(\rhot_n) = 0
\]
for $0 \le n \le k$.  Differentiating with respect to $\rhot_i$ shows that
\[
\sum_{j=0}^{2k+1} \frac{\partial \alpha_j}{\partial \rhot_i} q_j(\rhot_n) + \sum_{j=0}^{2k+1} \alpha_j q_j'(\rhot_n) \frac{\partial \rhot_n}{\partial \rhot_i}= 0
\]
It follows that
\[
\sum_{j=0}^{2k+1} \frac{\partial \alpha_j}{\partial \rhot_i}(\rho) \, q_j(\rho_n)  = 0, 
\]
because $\partial \rhot_0/\partial \rhot_i$ vanishes at $\rho$ while for $1
\le n \le k$,
\[
\sum_{j=0}^{2k+1} \alpha_j q_j'(\rhot_n) = 0.
\]
We have therefore found $k+1$ linearly independent vectors in the kernel of
left multiplication by $T$, as desired.
\end{proof}

\begin{proof}[Proof of Proposition~\ref{proposition:summation-formula}]
By Lemma~\ref{lemma:formula-exists}, a summation formula exists, and all that
remains is to prove uniqueness and the sign conditions.

Because $M(\rho)$ is nonsingular, the values $g(0)$ and $g(\rho_i)$ with $1
\le i \le k$ can be chosen arbitrarily.  Thus, the summation formula must be
unique up to scaling, and the coefficient $c_0$ of $\rho_0$ cannot vanish.

Now let $1 \le i \le k$, and let $\rhot$ equal $\rho$ except in the $i$-th
coordinate, where $\rhot_i = \rho_i+\varepsilon$ with $\varepsilon>0$ small.
Then $p_{\rhot}(\rho_i)$ and $p_{\rhot}(\rho_0)$ have opposite signs because
$r(p_{\rhot}) > r(p_{\rho})$, while $p_{\rhot}$ vanishes at the rest of
$\rho_1,\dots,\rho_k$. It follows from taking $g = p_{\rhot}$ that $c_i$ must
be nonzero, with the same sign as $c_0$.

Finally, when $s=1$ we can compute the sign of $c_{k+1}$ by taking $g=q_0=1$
to obtain
\[
\sum_{i=0}^{k+1}c_i = 0. \qedhere
\]
\end{proof}

When $s=-1$, we conjecture that $c_{k+1}$ always has the same sign as
$c_0,\dots,c_k$.  This conjecture holds for every case listed in
Table~\ref{table:computed}.

\section*{Acknowledgments}

We thank Noam Elkies and the anonymous referees for their helpful comments on
the manuscript, and David de~Laat for carrying out computations by a
different method that clarified the behavior of our computations for $d \le
4$.

\end{document}